\documentclass[11pt]{amsart}
\usepackage[T1]{fontenc}
\usepackage{amsaddr}
\usepackage{amsmath}%
\usepackage{amsthm}
\usepackage{amsxtra}%
\usepackage{amsfonts}%
\usepackage{amssymb}%
\usepackage[margin=1.25in]{geometry}
\usepackage{color,hyperref}
\usepackage{graphicx}
\usepackage{subfig}
\usepackage{cite}
\hypersetup{colorlinks,breaklinks,
             linkcolor=blue,urlcolor=blue,
             anchorcolor=blue,citecolor=blue}
\usepackage{color}
\usepackage{array}
\usepackage{booktabs}

\DeclareMathOperator*{\argmin}{arg\,min}
\renewcommand{\div}{\text{div}}

\newcommand{\A}{{\mathcal A}}

\renewcommand{\P}{{\mathbb P}}

\renewcommand{\bar}[1]{{\overline{#1}}}

\newcommand{\R}{\mathbb{R}}

\newcommand{\Z}{\mathbb{Z}}
\newcommand{\vb}[1]{\mathbf{#1}}

\newcommand{\eps}{\varepsilon}

\renewcommand{\phi}{\varphi}

\def\XXint#1#2#3{{\setbox0=\hbox{$#1{#2#3}{\int}$ }
\vcenter{\hbox{$#2#3$ }}\kern-.6\wd0}}

\newtheorem{theorem}{Theorem}
\newtheorem{lemma}[theorem]{Lemma}

\theoremstyle{definition}
\newtheorem{remark}[theorem]{Remark}

\graphicspath{{./images/}}

\begin{document} 
\title[PDE acceleration for obstacle problems]{PDE Acceleration: A convergence rate analysis and applications to obstacle problems}
\thanks{Jeff Calder was supported of NSF-DMS grant 1713691.  Anthony Yezzi was supported by NSF-CCF grant 1526848 and ARO W911NF-18-1-0281.\\
{\bf Source code:} \url{https://github.com/jwcalder/MinimalSurfaces}}
\author{Jeff Calder}
\address{School of Mathematics, University of Minnesota}
\email{jcalder@umn.edu}

\author{Anthony Yezzi}
\address{School of Electrical and Computer Engineering, Georgia Institute of Technology}
\email{anthony.yezzi@ece.gatech.edu}

\begin{abstract}
This paper provides a rigorous convergence rate and complexity analysis for a recently introduced framework, called \emph{PDE acceleration}, for solving problems in the calculus of variations, and explores applications to obstacle problems. PDE acceleration grew out of a variational interpretation of momentum methods, such as Nesterov's accelerated gradient method and Polyak's heavy ball method, that views acceleration methods as equations of motion for a generalized Lagrangian action.  Its application to convex variational problems yields equations of motion in the form of a damped nonlinear wave equation rather than nonlinear diffusion arising from gradient descent. These \emph{accelerated PDE's} can be efficiently solved with simple explicit finite difference schemes where acceleration is realized by an improvement in the CFL condition from $dt\sim dx^2$ for diffusion equations to $dt\sim dx$ for wave equations. In this paper, we prove a linear convergence rate for PDE acceleration for strongly convex problems, provide a complexity analysis of the discrete scheme, and show how to optimally select the damping parameter for linear problems. We then apply PDE acceleration to solve minimal surface obstacle problems, including double obstacles with forcing, and stochastic homogenization problems with obstacles, obtaining state of the art computational results.
\end{abstract}

\maketitle
\section{Introduction}

Optimization is one of the most prominent computational problems in science and engineering. For large scale problems, which are common in machine learning, second order methods, such as Newton's method, are intractable, and first order optimization algorithms are the method of choice \cite{bottou2010large}. One of the oldest first order algorithms for optimization is gradient descent
\begin{equation}\label{eq:graddesc}
x_{k+1} = x_k - \alpha \nabla f(x_k).
\end{equation}
Gradient descent converges reliably for convex problems, and versions of gradient descent (such as stochastic gradient descent) are state of the art in modern large scale machine learning problems \cite{bottou2010large}.

While gradient descent is a reliable first order method, for many problems it is very slow to converge. This has led to the development of accelerated versions of gradient descent that incorporate some form of momentum. One example is Polyak's heavy ball method \cite{polyak1964some}
\begin{equation}\label{eq:polyak}
x_{k+1} = x_k - \alpha \nabla f(x_k) + \beta (x_k - x_{k-1}).
\end{equation}
The term $\beta (x_k - x_{k-1})$ is referred to as momentum, and acts to accelerate convergence. Polyak's heavy ball method is simply a discretization of the second order ODE
\begin{equation}\label{eq:polyakode}
\ddot{x} + a \dot{x} = -\nabla f(x),
\end{equation}
which corresponds to the equations of motion for a body in a potential field. This continuum version is also called \emph{heavy ball with friction} and was studied by Attouch, Goudou, and Redont~\cite{attouch2000heavy}, and also by Goudou and Munier \cite{goudou2009gradient}. Another example of a momentum descent algorithm is Nesterov's famous accelerated gradient descent~\cite{nesterov1983method}, one form of which is
\begin{equation}\label{eq:nes}
x_{k+1} = y_k - \alpha \nabla f(y_k),  \ \  y_{k+1} = x_{k+1} + \frac{k-1}{k+2}(x_{k+1}-x_{k}).
\end{equation}
Nesterov's accelerated gradient descent contains an initial step of gradient descent, and then a second momentum step that averages the new update with the previous iterate. Nesterov \cite{nesterov1983method} proved that the method converges (for strongly convex problems) at a rate of $O(1/t^2)$ after $t$ steps, which is optimal for first order methods. 

Many variants of Nesterov acceleration have been proposed over the years, and the methods are very popular in machine learning \cite{sutskever2013importance,wibisono2016variational}, due to both the acceleration in convex problems, and the ability to avoid local minima in nonconvex problems. Recent work has begun to shed light on the fundamental nature of acceleration in optimization. Su, Boyd and Candes \cite{su2014differential} showed that Nesterov acceleration is a discretization of the second order ODE
\begin{equation}\label{eq:contNest}
\ddot{x} + \frac{3}{t}\dot{x} = -\nabla f(x).
\end{equation}
This ODE has been termed \emph{continuous time Nesterov} \cite{ward2017toolkit}.
Since the friction coefficient $3/t$ vanishes as $t\to \infty$, many implementations of the algorithm involve \emph{restarting}, whereby time is reset to $t=0$ whenever the system appears underdamped \cite{ward2017toolkit}. Wibisono, Wilson, and Jordan \cite{wibisono2016variational} went further, showing that all Nesterov type accelerated descent methods can be realized as discretizations of equations of motion in a generalized Lagrangian sense. In doing so, they offer a highly insightful and useful variational characterization of accelerated gradient descent.

Following their Lagrangian formulation, Yezzi and Sundaramoorthi \cite{yezzi2017accelerated} developed an accelerated PDE framework for solving active contour models in image segmentation, which are notorious for local minima. In a parallel work, Sundaramoorthi and Yezzi \cite{sundaramoorthi2018accelerated} (see also \cite{sundaramoorthi2018accelerated-nips}) applied the same ideas to flows of diffeomorphisms, which have applications in computer vision, such as optical flow problems. This \emph{PDE acceleration} framework was further developed by Benyamin, Calder, Sundaramoorthi and Yezzi \cite{benyamin2018anaccelerated} in the context of calculus of variations problems defined for functions on $\R^n$, including stability analysis for various explicit and semi-implicit discretization schemes, where they illustrated several examples in image processing such as total variation (TV) and Beltrami regularization. They drew special attention to a general class of regularized optimization problems where the \emph{accelerated PDE} takes the form a damped nonlinear \emph{wave} equation (generalizing \eqref{eq:polyakode} and \eqref{eq:contNest}), and the acceleration is realized as an improvement in the CFL condition from $dt\sim dx^2$ for diffusion equations (or standard gradient descent), to $dt\sim dx$ for wave equations.  We also mention that there have been some recent approaches to acceleration in image processing, which involve solving PDEs arising from variational problems \cite{grewenig2016fsi,hafner2016fsi,weickert2016cyclic,bahr2017fast}. Since these methods are not derived from a variational (Lagrangian) perspective, the methods do not descend on an energy and lack convergence guarantees and convergence rates.  Finally, we mention an interesting recent work \cite{schaeffer2016accelerated} that considers acceleration-type schemes for non-linear elliptic equations that \emph{do not} arise from variational formulations. The authors of \cite{schaeffer2016accelerated} formulate their schemes to blend together acceleration and gradient descent in such a way that the iterates satisfy a comparison principle, which is then used to prove convergence to steady state. While the application is not variational, the authors observe acceleration similar to the variational setting. 

This paper has several contributions. First, we analyze \emph{PDE acceleration} and prove convergence with a linear rate for strongly convex problems. Using the convergence rate, we show that the computational complexity of PDE acceleration is $O(m^{n+1})$ for solving  a PDE in dimension $n$ on a grid with $m$ points along each coordinate axis (in other words, $O(N^{(n+1)/n})$ where $N=m^n$ is the number of grid points). This is the same complexity as the conjugate gradient method for linear problems \cite{leveque2007finite}. Second, we provide a linear analysis of PDE acceleration, and show how to optimally select the damping coefficient via the solution of an eigenvalue problem. As a toy example, we study the Dirichlet problem and show that PDE acceleration compares favorably to preconditioned conjugate gradient and MINRES methods \cite{leveque2007finite}. In contrast to other indirect methods, the PDE acceleration method is very simple to implement with explicit or semi-implicit Euler discretizations of the wave equation (discussed extensively in \cite{benyamin2018anaccelerated}), and extends directly to nonlinear problems.

Finally, we apply the PDE acceleration method to efficiently solve minimal surface obstacle problems \cite{wang2008algorithm,caffarelli1998obstacle,zosso2017efficient}. Solving obstacle problems requires resolving a \emph{free boundary}, which makes efficient solutions challenging to obtain. Many algorithms have been proposed for solving classes of obstacle problems; a short list includes penalty methods \cite{tran20151,brezis1968methodes,scholz1984numerical}, splitting and projection algorithms \cite{lions1979splitting,zhang2001multilevel}, free boundary formulations \cite{braess2007convergence,johnson1992adaptive,majava2004level}, Lagrange multipliers \cite{hintermuller2002primal,hintermuller2011obstacle},  domain decomposition \cite{badea2003convergence} and multigrid methods \cite{hoppe1987multigrid,tai2003rate}.  Of particular interest is a recent primal dual approach to obstacle problems \cite{zosso2017efficient}, which has some flavor of a momentum-based descent algorithm. The authors of \cite{zosso2017efficient} show that their primal dual approach for obstacle problems is significantly faster than existing approaches. As an independent contribution we make an improvement to the primal dual algorithm, allowing it to work for nonlinear minimal surface problems, and we compare the method to PDE acceleration. We find PDE acceleration is approximately 10x faster in terms of computation time in C code for most experiments, with the difference attributed to the non-explicit dual update in \cite{zosso2017efficient}.  We also compare against the $L^1$-penalty method of \cite{tran20151}, which we find is significantly slower than both primal dual and PDE acceleration for nonlinear obstacle problems. 

We mention that, at the discrete level, PDE acceleration resembles other momentum based algorithms, such as the heavy ball method or Nesterov acceleration \cite{polyak1964some,nesterov1983method}. The results in this paper show that there are significant advantages to formulating a general continuum PDE acceleration framework. First, by performing the convergence rate analysis at the PDE level, we get a mesh-free convergence rate and the number of iterations to converge depends solely on the CFL time step restriction. Second, the parameters in the model---the friction coefficient and time step---can now be chosen optimally from PDE considerations, and do not require manual fine-tuning. In particular, the optimal choice for the damping/friction coefficient can be derived from an eigenvalue problem (see Section \ref{sec:optdamp}), while the largest stable time step is determined from the CFL condition \cite{leveque2007finite}.

\subsection{Outline}
This paper is organized as follows. In Section \ref{sec:frame} we review (for the case of functions defined over
$\R^n$) and slightly generalize the PDE acceleration framework, prove a linear convergence rate, and analyze the complexity of PDE acceleration. In particular, in Section \ref{sec:optdamp} we show how to select the damping coefficient optimally for linear problems. In Section \ref{sec:DP} we study the Dirichlet problem as a toy example, and explore connections to primal dual algorithms. In Section \ref{sec:obstacleproblems} we show how to apply PDE acceleration to nonlinear obstacle problems, and describe our improved version of the primal dual method from \cite{zosso2017efficient}. Finally, in Section \ref{sec:app} we give results of numerical simulations comparing PDE acceleration to primal dual and $L^1$-penalty methods for several different obstacle problems, including double obstacle problems with forcing, and stochastic homogenization problems with obstacles.

\section{PDE acceleration framework}
\label{sec:frame}

We review here the PDE acceleration framework for solving unconstrained problems in the calculus of variations for functions over $\R^n$, as originally presented in \cite{benyamin2018anaccelerated}. Consider the general unconstrained calculus of variations problem
\begin{equation}\label{eq:CV}
\min_{u\in \A}E[u]:=\int_\Omega L(x,u,\nabla u)\, dx,
\end{equation}
where $\Omega \subset \R^n$ and $\A =g +H^1_0(\Omega)$ or $\A = H^1(\Omega)$. We write $L=L(x,z,p)$ and write $\nabla_x L$, $L_z$, and $\nabla_pL$ for the partial derivatives of $L$ in each variable. There is no loss of generality in considering the unconstrained problem since we will handle constraints (such as obstacles) later with an $L^2$-penalty term (see Section \ref{sec:pdeacc}). We define the generalized action integral
\begin{equation}\label{eq:ai}
J[u] = \int_{t_0}^{t_1}k(t)\left( K[u] - b(t)E[u]\right) \, dt,
\end{equation}
where $u=u(x,t)$,  $k(t),b(t)$ are time-dependent weights, and $K[u]$ is the analog of kinetic energy, which we take to be 
\begin{equation}\label{eq:kin}
K[u] = \frac{1}{2}\int_\Omega \rho(x,u,\nabla u) u_t^2\, dx,
\end{equation}
where $\rho:\Omega\times \R\times \R^n\to \R_+$ is a mass density that may depend on $u$ and $\nabla u$. The action $J$ is a Lagrangian action with kinetic energy $K$ and potential energy $E$. We note that the time-dependent weight $k(t)$ is necessary to ensure dissipation of energy (in particular dissipation of the objective $E$; see Lemma \ref{lem:monotone}). Regarding the mass density, often one may take $\rho=\rho(x)$ or $\rho=1$. The more general setting may be useful, for example, in level set problems like image segmentation, where the object of interest is really the zero level set of $u$, and the kinetic energy of the zero level set can be obtained by selecting $\rho(x,z,p) = \delta(z)/|p|$.

The descent equations for the PDE acceleration method are exactly the Euler-Lagrange equations for $J$, i.e., the equations of motion, which we derive now. Let us write
\begin{equation}\label{eq:gE}
\nabla E[u] := L_z(x,u,\nabla u) -\div\left( \nabla_p L(x,u,\nabla u)\right)
\end{equation}
and
\begin{equation}\label{eq:gK}
\nabla K[u] := \frac{1}{2}u_t^2\rho_z(x,u,\nabla u) -\frac{1}{2}\div\left(u_t^2 \nabla_p \rho(x,u,\nabla u)\right)
\end{equation}
for the Euler-Lagrange equation for $E$ and for $K$. We recall that $\nabla E$ can be interpreted as the gradient of $E$ in the sense that
\begin{equation}\label{eq:var}
\frac{d}{d\eps}\Big\vert_{\eps=0}E[u + \eps v] = \int_\Omega \nabla E[u] \,v \, dx
\end{equation}
for all $v$ smooth with compact support in $\Omega$. Using this identity, a  variation on $J$ yields
\begin{align*}
\frac{d}{d\eps}\Big\vert_{\eps=0}J[ u+\eps v ] &= \int_{t_0}^{t_1}\int_{\Omega} k(t)\rho u_tv_t + k(t)\nabla K[u]\,v -k(t)b(t)\nabla E[u]\,v\, dx,  \\
&=\int_{t_0}^{t_1}\int_{\Omega} \left(-\frac{d}{dt}\left(k(t)\rho u_t\right) + k(t)\nabla K[u]- k(t)b(t) \nabla E[u]\right)v\, dx, 
\end{align*}
for $v \in C^\infty_c(\Omega\times (t_0,t_1))$. Therefore, the equations of motion are
\begin{equation}\label{eq:motion}
\frac{d}{dt}\left( \rho u_t \right) + a(t)\rho u_t =\nabla K[u] -b(t) \nabla E[u],
\end{equation}
where $a(t) = k'(t)/k(t)$. The nonlinear wave equation \eqref{eq:motion} is the descent equation for PDE acceleration, and the minimizer of $E$ is obtained by solving the equation for some initial conditions $u(x,0)$ and $u_t(x,0)$, and sending $t\to \infty$. The boundary condition depends on the choice of $\A$---we either have $u=g$ or $\nabla_pL\cdot \vb{n}=0$ on $\partial\Omega$, where $\vb{n}$ is the unit outward normal to $\partial \Omega$. Of course, we can also consider problems with mixed boundary conditions.

Notice in the equations of motion \eqref{eq:motion} the gradient $-\nabla E$ is now a forcing term in a damped wave equation, so it contributes to a change in velocity at each time step. The reader should contrast this with gradient descent
\[u_t = -\nabla E[u],\]
where the gradient is \emph{exactly} the velocity term, which can change \emph{instantaneously}.

In special cases, we can recover continuum versions of Polyak's heavy ball method, and Nesterov acceleration. For example, if we take $\rho=1$ and $a(t)$ and $b(t)$ to be constants, we get the PDE continuum version of the heavy ball with friction \eqref{eq:polyakode}, and for $a(t)=3/t$, $\rho=1$ and $b(t)=1$ we get the continuum version of Nesterov acceleration. 

\subsection{Convergence rate}
\label{sec:rate}

We prove in this section a convergence rate for the solution $u$ of the equations of motion \eqref{eq:motion} to the steady state solution $u^*$ of 
\begin{equation}\label{eq:steady}
\nabla E[u^*] = 0 \ \ \text{ in }\Omega,
\end{equation}
subject to the Dirichlet condition $u^*=g$ or the Neumann-type condition $\nabla_p L\cdot \vb{n}=0$ on $\partial\Omega$.  We assume throughout this section that $\Omega$ is open and bounded with Lipschitz boundary $\partial\Omega$.  

We first establish monotonicity of total energy.
\begin{lemma}[Energy monotonicity]\label{lem:monotone}
Assume $a(t),b(t)\geq 0$ and let $u$ satisfy \eqref{eq:motion}. Suppose either $u(x,t)=g(x)$ or $\nabla_pL(x,u,\nabla u)\cdot \vb{n}=0=\nabla_p\rho(x,u,\nabla u)\cdot \vb{n}$ on $\partial\Omega$. Then 
\begin{equation}\label{eq:monotone}
\frac{d}{dt}\left( K[u] + b(t)E[u] \right) = -2a(t)K[u] + b'(t)E[u].
\end{equation}
\end{lemma}
\begin{proof}
First note that
\[\frac{d}{dt}E[u] = \frac{d}{d\eps}\Big\vert_{\eps=0}E[u + \eps u_t] = \int_\Omega \nabla E[u]u_t\, dx + \int_{\partial\Omega} u_t \nabla _pL(x,u,\nabla u)\cdot \vb{n}\, dS.\]
Due to the boundary condition, either $u_t=0$ or $\nabla _pL(x,u,\nabla u)\cdot \vb{n}=0$ on $\partial\Omega$. Therefore
\begin{equation}\label{eq:dKT}
\frac{d}{dt}E[u]=\int_\Omega \nabla E[u]u_t\, dx.
\end{equation}
Similarly, we have
\begin{equation}\label{eq:dET}
\frac{d}{dt}K[u]=\int_\Omega \rho u_t u_{tt} +  \nabla K[u]u_t\, dx.
\end{equation}
Using the equations of motion \eqref{eq:motion} we also have
\begin{align*}
\frac{d}{dt}K[u]&=\frac{1}{2}\int_\Omega \rho u_t u_{tt}  + u_t \frac{d}{dt}(\rho u_t) \, dx\\
&=\frac{1}{2}\int_\Omega \rho u_t u_{tt} + \nabla K[u]u_t\, dx - \frac{1}{2}\int_\Omega b(t)\nabla E[u]u_t + a(t)\rho u_t^2\, dx\\
&=\frac{1}{2} \frac{d}{dt}K[u] - \frac{1}{2}b(t)\frac{d}{dt}E[u] - a(t)K[u],
\end{align*}
where we used \eqref{eq:dKT} and \eqref{eq:dET} in the last line. It follows that
\[\frac{d}{dt}K[u] + b(t)\frac{d}{dt}E[u] = -2a(t)K[u].\]
The proof is completed by adding $b'(t)E[u]$ to both sides.
\end{proof}
\begin{remark}
If $b'(t)\leq 0$ and $E[u]\geq 0$, then the total energy $K[u] + b(t)E[u]$ is monotonically decreasing at a rate controlled by the damping coefficient $a(t)$. In particular, we have
\begin{equation}\label{eq:mon2}
\frac{d}{dt}\left( K[u] + b(t)E[u] \right)\leq -2a(t)K[u].
\end{equation}
\end{remark}

We now prove a linear convergence rate in the special case that $\rho,a$ and $b$ are constants, $E$ has the form
\begin{equation}\label{eq:Eform}
E[u] = \int_\Omega \Phi(x,\nabla u) + \Psi(x,u)\, dx,
\end{equation}
and $u$ and $u^*$ satisfy the Dirichlet condition $u=g=u^*$ on $\partial\Omega$. We assume that $\Phi,\Psi\in C^2$, $\Phi=\Phi(x,p)$ is convex in $p$, $\Psi=\Psi(x,z)$ is convex in $z$, and for all $p\in \R^n$, $x\in \Omega$, and $z\in \R$
\begin{equation}\label{eq:ctheta}
\theta I \leq \nabla^2_p \Phi(x,p) \leq \theta^{-1}I
\end{equation}
and
\begin{equation}\label{eq:cmu}
\Psi_{zz}(x,z) \leq \mu
\end{equation}
for some $\theta,\mu >0$. 
\begin{remark}

If there exists $M>0$ such that $|\nabla u^*|\leq M$, where $u^*$ solves \eqref{eq:steady}, then we can relax \eqref{eq:ctheta} to the condition that $\nabla^2_p\Phi>0$ (that is, $\Phi$ is strictly convex in $p$). Indeed, define
\begin{equation}\label{eq:phinew}
\bar{\Phi}(x,p) := \Phi(x,p)\varphi(p) + K\max\{|p|-M,0\}^2,
\end{equation}
where $K>0$ and $\varphi\in C^\infty$ is a bump function with $0\leq \varphi\leq 1$, $\varphi(p)=1$ for $|p|\leq 2M$, and $\varphi(p)=0$ for $|p|\geq 4M$. Clearly $ \bar{\Phi}(x,p)$ satisfies \eqref{eq:ctheta} for some $\theta>0$ when $|p|\geq 4M$, and by choosing $K$ large enough, the strict convexity of $\Phi$ and compactness of $\bar{\Omega} \times B(0,2M)$ allow us to extend the condition \eqref{eq:ctheta} to all $(x,p) \in \Omega\times \R^n$. Now let $\bar{E}$ be the energy in \eqref{eq:Eform} with $\bar{\Phi}$ in place of $\Phi$. Since $|\nabla u^*|\leq M$ and $\bar{\Phi}(x,p) = \Phi(x,p)$ for all $|p|\leq M$, we see that $\nabla \bar{E}(u^*)=0$; that is, $u^*$ is the unique solution of $\nabla \bar{E}=0$ as well. We can use $\bar{E}$ in place of $E$ in PDE acceleration to ensure \eqref{eq:ctheta} holds while obtaining the same steady state solution (though the dynamics can be different). This is mainly a theoretical concern, and not something we do in practical applications.

We note that for minimal surface obstacle problems, the gradients of solutions are H\"older continuous (e.g., $u^*\in C^{1,\alpha}$) \cite{caffarelli1998obstacle}, and hence we can always find such an $M$.
\end{remark}
From now on we write $\nabla^2$ in place of $\nabla ^2_p$.

We now prove the following linear convergence rate.
\begin{theorem}[Convergence rate]\label{thm:rate}
Let $u$ satisfy \eqref{eq:motion}. Assume \eqref{eq:Eform}, \eqref{eq:ctheta}, and \eqref{eq:cmu} hold, $u=u^*$ on $\partial\Omega$, $a(t)=a>0$ is constant and $b(t)\equiv 1$ and $\rho\equiv 1$. Then there exists $C>0$ depending on $a,\theta$, $u(x,0)$, and $u_t(x,0)$ such that
\begin{equation}\label{eq:rate}
\|u - u^*\|_{H^1(\Omega)}^2 \leq C\exp\left( -\beta t \right),
\end{equation}
where
\begin{equation}\label{eq:beta}
\beta = \frac{a\sqrt{c^2+4\lambda\theta}-ac}{2\sqrt{\lambda\theta}+a}\ \ \text{with } c = a + \frac{\mu}{a}+\frac{2\lambda}{a}(\theta^{-1}-\theta),
\end{equation}
and  $\lambda>0$ is the Poincar\'e constant for $\Omega$.
\end{theorem}
\begin{proof}
We use energy methods. Since $u$ solves \eqref{eq:motion} we have
\[u_{tt} + au_t + \nabla E[u] - \nabla E[u^*] = 0.\]
Multiply both sides by $w:=u-u^*$ and integrate over $\Omega$ to find
\[\int_\Omega w_{tt}w + aw_tw + (\nabla E[u] - \nabla E[u^*])(u-u^*)\,dx = 0.\]
Integrating by parts we have
\begin{align*}
&\int_\Omega (\nabla E[u] - \nabla E[u^*])(u-u^*)\,dx\\
& =\int_\Omega  (\nabla \Phi(\nabla u) - \nabla \Phi(\nabla u^*))\cdot (\nabla u-\nabla u^*) + (\Psi_z(x,u) - \Psi_z(x,u^*))(u-u^*)\, dx.
\end{align*}
Since $\Psi$ is convex in $z$ and $\Phi$ is strongly convex (by \eqref{eq:ctheta}) we deduce
\[\int_\Omega (\nabla E[u] - \nabla E[u^*])(u-u^*)\,dx \geq \theta\int_\Omega |\nabla w|^2\, dx.\]
It follows that
\[\int_\Omega w_{tt}w + aw_tw + \theta|\nabla w|^2\,dx \leq 0,\]
and so
\begin{equation}\label{eq:ed}
\frac{d}{dt}\int_\Omega \frac{1}{2}a^2w^2 + aww_t\, dx \leq \int_\Omega aw_t^2 - a\theta |\nabla w|^2\, dx = 2aK[w] - a\theta \int_\Omega |\nabla w|^2\,dx.
\end{equation}
This leads us to define the energy
\begin{equation}\label{eq:ene}
e(t):= \int_\Omega \frac{1}{2}a^2w^2 + aww_t \, dx + 2(K[u] + E[u]-E[u^*]).
\end{equation}
Note first that 
\[e(t) = \int_\Omega \frac{1}{2}(aw+w_t)^2 + \frac{1}{2}w_t^2 \, dx+ 2(E[u]-E[u^*]) \geq 0,\]
so $e$ is a valid energy. Using Lemma \ref{lem:monotone} and \eqref{eq:ed} we have
\begin{equation}\label{eq:edot}
\dot{e}(t) \leq 2aK[w] - 4aK[w] - a\theta \int_\Omega |\nabla w|^2\, dx = -a\int_\Omega \theta|\nabla w|^2\, + w_t^2 dx.
\end{equation}

Now, we compute
\begin{align*}
E[u] - E[u^*]&=\int_0^1 \frac{d}{dt}E[u^* + t(u-u^*)]\, dt\\
&=\int_0^1 \int_0^s\frac{d^2}{dt^2}E[u^* + t(u-u^*)]\, dt\, ds\\
&=\int_0^1 \int_0^s\int_\Omega\sum_{i,j=1}^n \Phi_{p_ip_j}(\nabla u^*+t(\nabla u-\nabla u^*))(u_{x_i}-u^*_{x_i})(u_{x_j}-u^*_{x_j}) \\
&\hspace{1.5in} + \Psi_{zz}(x,u^*+t(u-u^*))(u-u^*)^2\, dx\, dt\, ds\\
&\leq \int_0^1\int_0^s\int_\Omega \theta^{-1}|\nabla w|^2 + \mu w^2\,dx\, dt\,ds\\
&= \frac{1}{2}\int_\Omega \theta^{-1}|\nabla w|^2 + \mu w^2\,dx,
\end{align*}
where the last inequality used \eqref{eq:ctheta}. Combining this with the Poincar\'e inequality $\lambda\int_\Omega w^2 \,dx \leq  \int_\Omega|\nabla w|^2 \, dx$, we have
\begin{align*}
e(t) &\leq \frac{1}{2}\int_\Omega (a^2+a\eps + \mu)w^2 + a\eps^{-1} w_t^2 \, dx + 2(K[u] + E[u]-E[u^*])\\
&\leq  \int_\Omega \left[ \tfrac{1}{2}(a^2+a\eps+\mu)\lambda^{-1} + \theta^{-1} \right]|\nabla w|^2 + \frac{1}{2}\left(2 + a\eps^{-1}  \right)w_t^2 \, dx.
\end{align*}
Now, there is a unique value of $\eps>0$ such that
\begin{equation}\label{eq:epsval}
\tfrac{1}{2}(a^2+a\eps + \mu)\lambda^{-1} + \theta^{-1} = \frac{\theta}{2}(2 + a\eps^{-1}).
\end{equation}
Selecting this value for $\eps$ we have by \eqref{eq:edot} that
\[e(t) \leq \frac{1}{2}(2 + a\eps^{-1})\int_\Omega \theta|\nabla w|^2 + w_t^2\, dx\leq  -\left( \frac{1}{a}+\frac{1}{2\eps} \right)\dot{e}(t).\]
Therefore 
\[\dot{e}(t) \leq \frac{-2a\eps}{2\eps+a}e(t),\]
from which it follows that
\begin{equation}\label{eq:ratee}
e(t) \leq e(0)\exp\left( \frac{-2a\eps}{2\eps+a}t \right).
\end{equation}

All that remains is to compute $\eps$. Note that \eqref{eq:epsval} is equivalent to
\[\eps^2 + \left( a + \frac{\mu}{a}+\frac{2\lambda}{a}(\theta^{-1}-\theta) \right)\eps - \theta\lambda = 0,\]
and hence
\[2\eps = \sqrt{c^2 + 4\lambda\theta}-c \leq 2\sqrt{\lambda\theta},\]
where 
\[c = a + \frac{\mu}{a}+\frac{2\lambda}{a}(\theta^{-1}-\theta).\]
\end{proof}
\begin{remark}
The energy estimates obtained in the proof of Theorem \ref{thm:rate} establish uniqueness of solutions to \eqref{eq:motion} with values in $H^1(\Omega)$ and Dirichlet boundary conditions. We expect that existence of solutions follows from combining these energy estimates with a standard Galerkin approximation \cite{evans2002book}. Such a result is outside the scope of this paper, and we leave it to future work.
\end{remark}
\begin{remark}
Instead of formulating PDE acceleration at the continuum level and then discretizing the descent equations to compute the solution, it is possible to formulate PDE acceleration entirely in the discrete-space setting, by starting with a discretization of $E[u]$. Then the resulting descent equations become a discretization of \eqref{eq:motion} that is discrete in space and continuous in time, provided the discrete divergence is defined as the exact numerical adjoint of the discrete gradient (e.g., forward differences for the gradient and backward differences for the divergence, as we use in Section \ref{sec:app}). We expect the proof of Theorem \ref{thm:rate} to extend, with minor modifications, in this situation. 
\end{remark}

\subsection{Computational complexity}
\label{sec:complexity}

Theorem \ref{thm:rate} allows us to analyze the computational complexity of PDE acceleration. Suppose our domain is the unit box $[0,1]^2$ and we discretize the problem on an $m\times m$ grid with uniform spacing $dx=1/m$. Since we can discretize the wave equation \eqref{eq:motion}  with a time step $dt = O(dx) = O(1/m)$ while satisfying the CFL condition, the number of iterations required to converge to within a tolerance of $\eps>0$ in the $H^1$ norm satisfies
\begin{equation}\label{eq:kup}
k \leq c\beta^{-1}m\log(C\eps^{-1}).
\end{equation}
Using an explicit time stepping scheme, each iteration has complexity $O(m^2)$, hence the computational complexity for $2D$ problems is 
\begin{equation}\label{eq:2Dcomp}
\text{2D complexity }= O(\beta^{-1}m^3 \log(\eps^{-1})).
\end{equation}
A similar computation for $3D$ problems on a $m\times m\times m$ grid yields a complexity of 
\begin{equation}\label{eq:3Dcomp}
\text{3D complexity }= O(\beta^{-1}m^4 \log(\eps^{-1})).
\end{equation}
If we write the complexity in terms of the number of unknowns, which is $N=m^2$ in $2D$ and $N=m^3$ in $3D$, then the complexity is $O(N^{3/2})$ for 2D problems and $O(N^{4/3})$ for 3D problems. For linear problems, this complexity is similar to conjugate gradient methods \cite{leveque2007finite}.

\subsection{Optimal damping for linear problems}
\label{sec:optdamp}

While Theorem \ref{thm:rate} provides a convergence rate for the PDE acceleration method, it does not indicate how to optimally select the damping parameter $a$ to achieve the optimal rate. Here, we consider the selection of $a$ for linear problems. For nonlinear problems, we propose to linearize and apply the analysis described below.

 Let $L$ be a uniformly elliptic second order partial differential operator in divergence form, that is
\begin{equation}\label{eq:L}
Lu = -\sum_{i,j=1}^n (a^{ij}u_{x_i})_{x_j},
\end{equation}
where $A(x):=(a^{ij}(x))_{ij}\in C^\infty(\Omega)$, and there exists $\theta>0$ such that $A(x)\geq \theta I$ for all $x$. We consider the PDE acceleration method for solving the Dirichlet problem
\begin{equation}\label{eq:Lp}
\left\{\begin{aligned}
Lu^* + cu^* &= f&&\text{in }\Omega\\ 
u^* &=g&&\text{on }\partial\Omega,
\end{aligned}\right.
\end{equation}
where $c\geq 0$.
The equations of motion with constant $a(t)=a$ and $b(t)=b$ are
\begin{equation}\label{eq:Lpmotion}
\left\{\begin{aligned}
u_{tt}+au_t + bLu + bcu &= f&&\text{in }\Omega\times (0,\infty)\\ 
u &=g&&\text{on }\partial\Omega\times (0,\infty)\\
u &=u_0&&\text{on }\Omega\times \{t=0\}.
\end{aligned}\right.
\end{equation}
Let $w(x,t) = u(x,t) - u^*(x)$. Then $w$ satisfies
\begin{equation}\label{eq:Lpmotionw}
\left\{\begin{aligned}
w_{tt}+aw_t + bLw + bcw &= 0&&\text{in }\Omega\times (0,\infty)\\ 
w &=0&&\text{on }\partial\Omega\times (0,\infty)\\
w &=u_0&&\text{on }\Omega\times \{t=0\}.
\end{aligned}\right.
\end{equation}
We can expand the solution $w$ in a Fourier series
\begin{equation}\label{eq:fs}
w(x,t) = \sum_{k=1}^\infty d_k(t)v_k,
\end{equation}
where $v_1,v_2,\dots$ is an orthonormal basis for $L^2(\Omega)$ consisting of Dirichlet eigenfunctions of $L$ with corresponding eigenvalues 
\[0 <\lambda_1\leq \lambda_2\leq \lambda_3 \leq \cdots\]
That is, the function $v_k(x)$ satisfies
\begin{equation}\label{eq:eigen}
\left\{\begin{aligned}
Lv_k &= \lambda_k v_k &&\text{in }\Omega\\ 
v_k &=0&&\text{on }\partial\Omega.
\end{aligned}\right.
\end{equation}
Substituting \eqref{eq:fs} into \eqref{eq:Lpmotionw} we find that
\[d_k''(t) + ad_k'(t) + b(\lambda_k+c)d_k(t)=0.\]
The general solution is
\[d_k(t) = Ae^{r_{k,1}t}+Be^{r_{k,2}t},\]
where
\[r_{k,1} = -\frac{a}{2}+\frac{1}{2}\sqrt{a^2 - 4b(\lambda_k+c)}\text{ and } r_{k,2}=-\frac{a}{2}-\frac{1}{2}\sqrt{a^2 - 4b(\lambda_k+c)}.\]
Hence, the optimal decay rate is of the form $e^{-at/2}$ provided that
\[a^2 - 4b(\lambda_k+c) \leq 0 \text{ for all }k\geq  1.\]
This leads to the optimal choice for the damping coefficient
\begin{equation}\label{eq:aopt}
a = 2\sqrt{b(\lambda_1+c)}.
\end{equation}
With this choice of damping we have the convergence rate
\begin{equation}\label{eq:error}
|u(x,t)-u^*(x)|\leq C\exp\left( -\sqrt{b(\lambda_1+c)}\,t \right),
\end{equation}
for some constant $C>0$ depending on the initial condition $u_0$. 

We note that when $L$ is degenerate, so $\lambda_1=0$, the convergence rate is
\begin{equation}\label{eq:error2}
|u(x,t)-u^*(x)|\leq C\exp\left( -\sqrt{bc}\,t \right),
\end{equation}
with the optimal choice $a=2\sqrt{bc}$. In particular, if $c=0$ then the method does not converge, since there are undamped Fourier modes.

\section{Dirichlet problem}
\label{sec:DP}

As an illustrative example, we consider the Dirichlet problem
\begin{equation}\label{eq:DP}
\min\left\{ \frac{1}{2}\int_\Omega |\nabla u|^2\, dx \ :\ u\in H^1(\Omega) \text{ and }u=g \text{ on }\partial\Omega\right\}.
\end{equation}
Gradient descent corresponds to solving the heat equation
\begin{equation}\label{eq:heat}
\left\{\begin{aligned}
u_t -\Delta u &= 0&&\text{in }\Omega\times (0,\infty)\\ 
u &=g&&\text{on }\partial\Omega\times (0,\infty)\\
u &=u_0&&\text{on }\Omega\times \{t=0\},
\end{aligned}\right.
\end{equation}
while PDE acceleration corresponds to solving the damped wave equation
\begin{equation}\label{eq:dampwave}
\left\{\begin{aligned}
u_{tt}+au_t -b\Delta u &= 0&&\text{in }\Omega\times (0,\infty)\\ 
u &=g&&\text{on }\partial\Omega\times (0,\infty)\\
u &=u_0&&\text{on }\Omega\times \{t=0\}.
\end{aligned}\right.
\end{equation}
For concreteness, we consider the domain $\Omega:=[0,1]^2$.
Here, the first Dirichlet eigenvalue is $\lambda_1 = \pi^2$, and hence the optimal choice of the damping coefficient from \eqref{eq:aopt} is
\[a = 2\pi\sqrt{b}. \]
With this choice of $a$, the accelerated PDE method converges to the solution of the Dirichlet problem \eqref{eq:DP} at a rate  of $\exp(-b\pi t)$.
There are no parameters to select in the heat equation \eqref{eq:heat}. It is possible to show with a Fourier expansion that the solution of the heat equation \eqref{eq:heat} converges to the solution of the Dirichlet problem \eqref{eq:DP} at a rate of $\exp(-\pi^2 t)$. So far there is not much difference between the two equations---part of the difference comes from numerical stiffness, as explained below.

\subsection{Runtime-comparisons}

To solve both equations, we use the standard discretizations
\[u_t \approx \frac{u^{n+1}_{ij}-u^n_{ij}}{dt}, \  u_{tt}\approx\frac{u^{n+1}_{ij}-2u^n_{ij}+u^{n-1}_{ij}}{dt^2}, \text{ and}\]
\[\Delta u \approx \frac{u^n_{i+1,j}+u^n_{i-1,j}+u^n_{i,j+1}+u^n_{i,j-1}-4u^n_{ij}}{dx^2},\]
and explicit forward time stepping. The CFL condition for the damped wave equation is 
\[dt \leq \frac{dx}{\sqrt{2b}}.\]
By \eqref{eq:error}, the error decays like $\exp(-2\pi\sqrt{b} t)$. Therefore, to solve the problem to within a tolerance of $\eps$ we need $k$ iterations, where $k$ satisfies
\[\eps = C\exp\left( -2\pi \sqrt{b} kdt \right)=C\exp\left( -\sqrt{2}\pi kdx \right).\]
Hence, we need
\[k = \frac{1}{\sqrt{2}\pi dx}\log(C\eps^{-1})\]
iterations.  Note this is independent of $b$. Additionally, if we saturate the CFL condition and set $dt = dx/\sqrt{2b}$, then $b$ does not even appear in the numerical discretization of \eqref{eq:dampwave}. 

We contrast this with the heat equation \eqref{eq:heat}, where the CFL condition is $dt\leq dx^2/4$. Here, we need
\[k = \frac{4}{\pi^2 dx^2}\log(C\eps^{-1})\]
iterations for convergence. Table \ref{tab:DPsim} gives a comparison of the performance of PDE acceleration, gradient descent, and the primal dual algorithm from \cite{zosso2017efficient} for solving the Dirichlet problem on various grid sizes. We used the boundary condition $g(x_1,x_2) = \sin(2\pi x_1^2) + \sin(2\pi x_2^2)$ and ran each algorithm until the finite difference scheme was satisfied with an error of less than $dx^2$.  The initial conditions for both algorithms were $u(x,0)=g(x)$. For the primal dual algorithm~\cite{zosso2017efficient} we set $r_1=4\pi^2 r_2$, which is provably optimal using similar methods as in Section \ref{sec:optdamp}. We see that PDE acceleration is more than twice as fast as primal dual, while both significantly outperform standard gradient descent. We mention that our method converges to engineering precision very quickly (in about one fifth of the iteration count displayed in Table \ref{tab:DPsim}), while the majority of iterations are taken to resolve the solution up to the $O(dx^2)$ error (for the $1024\times 1024$ grid, the error tolerance is $dx^2 \approx 10^{-6}$). 
\begin{table}[!t]
\centering
\begin{tabular}{|c|c|c|c|c|c|c|}
 \hline
 &\multicolumn{2}{c|}{\textbf{Our Method}}&\multicolumn{2}{c|}{\textbf{Primal Dual~\cite{zosso2017efficient}}} &\multicolumn{2}{c|}{\textbf{Gradient Descent}}\\
\hline
Mesh & Time (s) & Iterations& Time (s) & Iterations & Time (s) & Iterations \\
\hline
$64^2$      &0.012&399   &0.02   &592    &0.148    &8404      \\
$128^2$     &0.05 &869   &0.11    &1384   &2.4      &3872     \\
$256^2$     &0.38  &1898  &1.0     &3027   &40       &174569     \\
$512^2$     &4.8  &4114  &13.1    &6831   &1032     &774606    \\
$1024^2$    &41   &8813  &115     &14674  &23391    &3399275   \\
\hline
\end{tabular}
\caption{Comparison of PDE acceleration, primal dual, and gradient descent for solving the Dirichlet problem. Runtimes are for C code.}
\label{tab:DPsim}
\end{table}

Of course, we do not recommend using PDE acceleration or primal dual methods for solving linear Poisson problems. In the linear setting there are faster algorithms available. For comparison, we show in Table \ref{tab:DPsim2} the runtimes for incomplete Cholesky preconditioned conjugate gradient and MINRES, Gauss-Seidel with successive overrelaxation, Matlab backslash, and the multigrid method with V-cycles. We see that PDE acceleration is comparable to preconditioned MINRES and conjugate gradient, while Matlab backslash (Cholesky factorization and triangular solve) and multigrid are significantly faster. We did not compare against FFT methods since they are specific to constant coefficient linear problems, which is rather restrictive, and would be comparable to multigrid and backslash. 

We mention that for general linear PDE, multigrid is normally much faster than Matlab backslash. The 2D Poisson equation is a special case where the linear system has a simple banded structure and direct solvers are highly efficient and comparable to multigrid. Moving to 3D problems one would expect multigrid to outperform backslash. Furthermore, our implementation of multigrid with V-cycles may not be optimal and further improvements could be possible. We emphasize that the Dirichlet problem is simply a toy illustrative example of PDE acceleration compared to gradient descent, and it is outside the scope of this paper to provide a thorough comparison to all linear solvers (e.g., other preconditioners, different multigrid cycling, etc.). Our real interest is in nonlinear problems with constraints. The linear methods that we compared against here do not extend directly to nonlinear problems, much less to obstacle constrained problems. The PDE acceleration method is formulated in the general nonlinear case, and is provably convergent with the same rate for nonlinear problems. In practice, we usually see similar computation times for nonlinear problems (see Section \ref{sec:app}).

\begin{table}[!t]
\centering
\begin{tabular}{|c|c|c|c|c|c|c|c|c|}
 \hline
 &\multicolumn{2}{c|}{\textbf{PCG}}&\multicolumn{2}{c|}{\textbf{MINRES}} &\multicolumn{2}{c|}{\textbf{Gauss-Seidel}}&\textbf{Backlash}&\textbf{Multigrid}\\
\hline
Mesh & Time & Iter.& Time & Iter. & Time & Iter. & Time &  Time\\
\hline
$64^2$     &0.027&56    &0.028   &54    &0.017    &197      & 0.013  & 0.036       \\
$128^2$    &0.08&120   &0.093    &114   &0.08      &432     & 0.037  & 0.048    \\
$256^2$    &0.67  &251  &0.78     &240   &0.72       &1020  & 0.12   & 0.13     \\
$512^2$    &7.9  &523  &9.4    &500   &6.1     &2046        & 0.61   & 0.55     \\
$1024^2$   &69   &1089  &81     &1044  &50    &4100         & 3      & 2.6        \\
\hline
\end{tabular}
\caption{ Runtimes in seconds for incomplete Cholesky preconditioned conjugate gradient, MINRES, Gauss-Seidel with successive overrelaxation, Matlab backslash, and V-cycle MultiGrid for solving the Dirichlet problem. The Gauss-Seidel method was implemented in C, while the other algorithms were implemented in Matlab. }
\label{tab:DPsim2}
\end{table}

\subsection{Initial condition}

We mention that the choice of initial condition can affect the computation time. If the initial condition does not continuously attain the boundary data, then fixing the boundary data on the first time step transfers a large amount of kinetic energy into the system that takes longer to dissipate. See Figure \ref{fig:energy} for a depiction of the kinetic, potential, and total energy for initial conditions $u(x,0)=g(x)$ and $u(x,0)=0$. The rate of convergence is not affected; it is just the constant in front, which corresponds to the initial energy, that is larger in this case. For example, in the simulation above, if we start from $u(x,0)=0$ on the $512\times 512$ grid, the computation takes 5529 iterations, or about roughly 1.4x more iterations compared to the initial condition $u(x,0) = g(x)$.
\begin{figure}
\centering
\subfloat[$u(x,0)=g(x)$]{\includegraphics[trim=20 0 30 20, clip = true, width=0.50\textwidth]{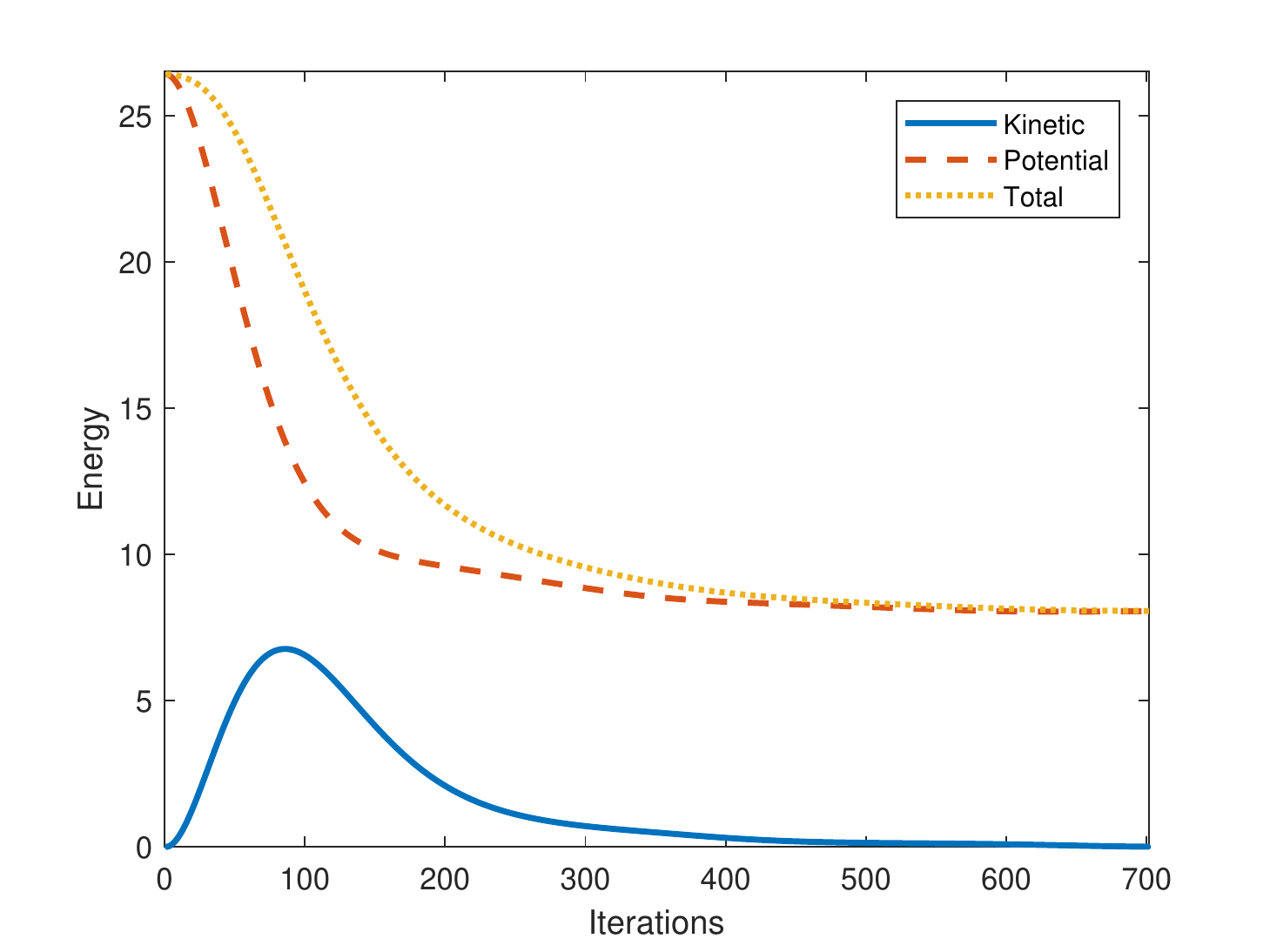}}
\subfloat[$u(x,0)=0$]{\includegraphics[trim=20 0 30 20, clip = true,width=0.50\textwidth]{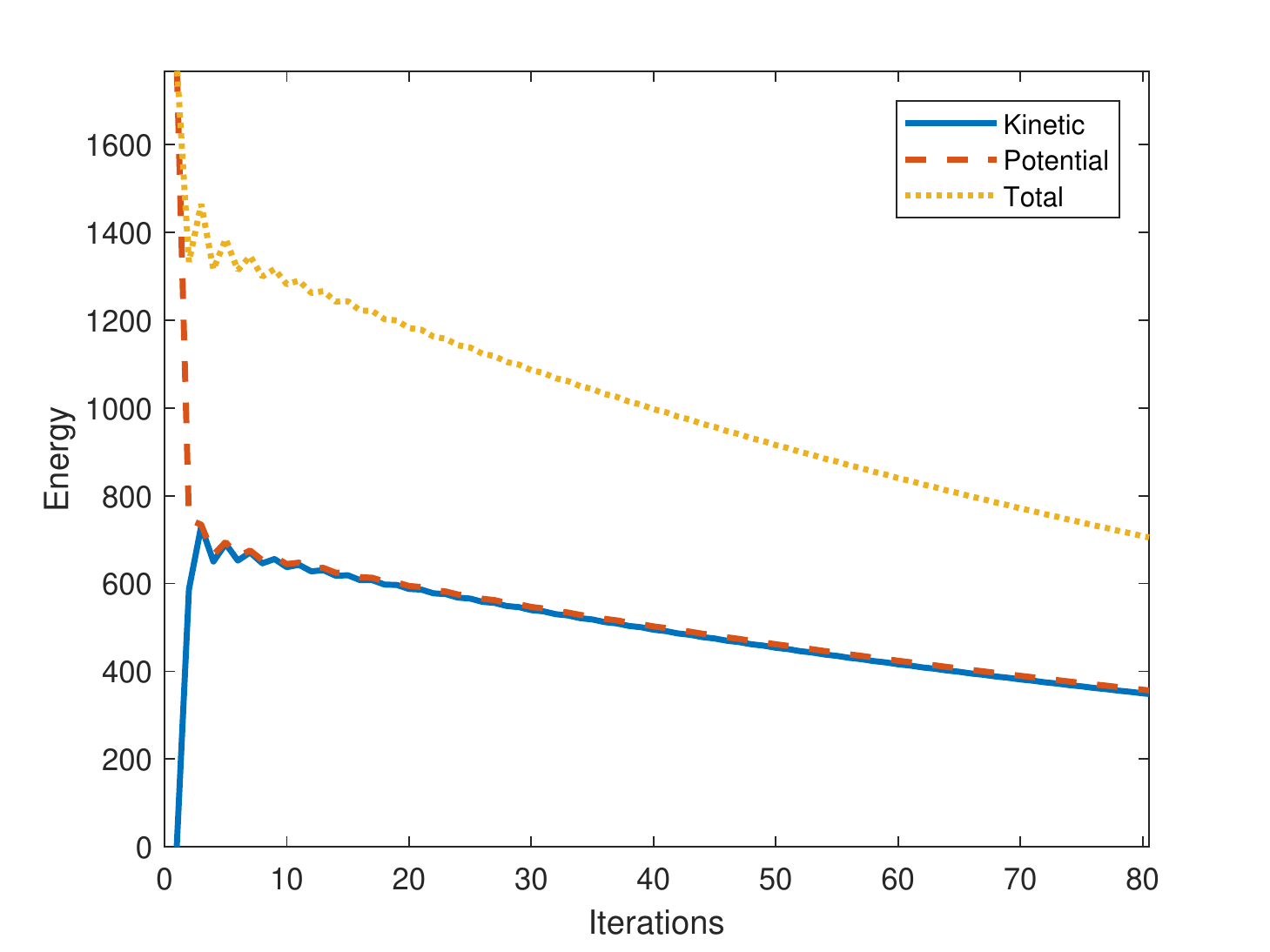}}
\caption{Comparison of energy dynamics for different initial conditions for solving the Dirichlet problem on a $512\times 512$ grid. When the initial condition does not continuously attain the boundary data, a nearly infinite amount of kinetic energy is transferred into the system at the first time step when the boundary conditions are set. This takes longer to dissipate and slows convergence.}
\label{fig:energy}
\end{figure}

This minor issue can be easily fixed in one of two ways. First, we can, if possible, choose an initial condition that continuously attains the boundary data. A second solution is to start from any arbitrary initial condition, and then change the boundary conditions gradually, instead of instantaneously. This can be done by gradient descent on the energy
\[I(u) = \frac{1}{2}\int_{\partial\Omega}(u-g)^2\,dS.\]
That is, on the boundary we solve the ordinary differential equation
\[u_t = g-u.\]
Both solutions give similar improvements in the speed of convergence in our simulations. To keep the algorithm simple, we chose not to implement either of these fixes in the rest of the paper. There are many other tricks that one can play with to speed up convergence, such as increasing the damping factor $a$ as a function of time, or incorporating multi-grid methods. We leave investigations along these lines to future work. 

\subsection{Connection to Primal Dual methods}
\label{sec:primaldual}

In the context of the Dirichlet problem, there is a close connection between primal dual methods~\cite{zosso2017efficient}, and PDE acceleration. This was explored briefly in~\cite{zosso2017efficient}, where it was observed that their primal dual algorithm for solving the Dirichlet problem can be interpreted as a numerical scheme for a damped wave equation. We go further here, and give a PDE interpretation of primal dual methods and show exactly how they are related to PDE acceleration for the Dirichlet problem.

We recall that the convex dual, or Legendre-Fenchel transform, of a function $\Phi:\R^n \to \R$ is 
\begin{equation}\label{eq:dual}
\Phi^*(p) = \max_{x \in \R^n}\{ x\cdot p - \Phi(x)\}.
\end{equation}
If $\Phi$ is convex, then by convex duality we have $\Phi^{**}:= (\Phi^*)^*=\Phi$. We assume $\Phi:\R^n \to \R$ is convex and consider for simplicity the problem
\begin{equation}\label{eq:cprob}
\min_{u}\int_\Omega \Phi(\nabla u)\, dx,
\end{equation}
subject to a Dirichlet boundary condition $u=g$. A primal dual algorithm for solving \eqref{eq:cprob} expresses $\Phi$ through its convex dual $\Phi^*$ giving the initially more looking complicated formation
\begin{equation}\label{eq:cprob2}
\min_{u}\max_{p}\int_\Omega p\cdot \nabla u - \Phi^*(p)\, dx.
\end{equation}
Here, $u:\Omega \to \R$ is the primal variable and $p:\Omega \to \R^n$ is the dual variable.  Given $p\cdot \vb{n}=0$ on $\partial\Omega$, we can integrate by parts to express the problem as
\begin{equation}\label{eq:cprob3}
\min_{u}\max_{p}\int_\Omega -u \,\div(p) - \Phi^*(p)\, dx.
\end{equation}
The primal dual algorithm in \cite{zosso2017efficient} solves \eqref{eq:cprob2} by alternating proximal updates on $p$ and $u$ until convergence (see Section \ref{sec:pdual}).   In the continuum this is equivalent to jointly performing gradient descent on $u$ and gradient ascent on $p$, which corresponds to the coupled PDEs
\begin{equation}\label{eq:pdpde}
\left\{\begin{aligned}
p_t  &= a(t)(\nabla u - \nabla \Phi^*(p))\\ 
u_t  &=\div(p).
\end{aligned}\right.
\end{equation}
The factor $a(t)$ is the ratio of the time steps between the proximal updates on $u$ and $p$ in the primal dual algorithm; in the notation of \cite{zosso2017efficient}, $a = r_1/r_2$. To the best of our knowledge, this PDE interpretation of primal dual algorithms is a new observation.  In particular, we use this observation to optimally set ratio $r_1/r_2$ for the primal dual method in Section \ref{sec:pdual}.

For the Dirichlet problem, $\Phi(p) =\Phi^*(p) = \frac{1}{2}|p|^2$ and \eqref{eq:pdpde} becomes
\begin{equation}\label{eq:pdpdedp}
\left\{\begin{aligned}
p_t  &= a(t)(\nabla u - p)\\ 
u_t  &=\div(p).
\end{aligned}\right.
\end{equation}
In this case we can eliminate the dual variable and we obtain the damped wave equation
\begin{equation}\label{eq:pdwave}
u_{tt} + a(t) u_t - a(t)\Delta u = 0.
\end{equation}
Contrasting this with \eqref{eq:dampwave}, we see the key difference between primal dual and PDE acceleration methods is that primal dual methods are unable to adjust the damping coefficient $a(t)$ independently of other terms in the PDE. 

This explicit connection between primal dual and PDE acceleration seems to be a coincidence for the Dirichlet problem and does not hold in any other case that we are aware of. In particular, it seems necessary that $\nabla \Phi^*(p)$ is linear in $p$ in order to convert the system \eqref{eq:pdpde} into a scalar wave equation in $u$. We can eliminate the primal variable by differentiating the equation for $p_t$ in \eqref{eq:pdpdedp} to obtain
\begin{equation}\label{eq:pdualpde}
p_{tt} + \left( aDF(p) + \frac{a'(t)}{a(t)} \right)p_t = a(t)\nabla \div(p),
\end{equation}
where $F(p):=\nabla \Phi^*(p)$. However, this is no longer a wave equation.

\section{Obstacle problems}
\label{sec:obstacleproblems}

Consider the standard obstacle problem
\begin{equation}\label{eq:obP}
\min_{u\in \A} E[u] = \int_{\Omega}\Phi(x,\nabla u)\, dx,
\end{equation}
where
\begin{equation}\label{eq:Aset}
\A = \left\{u \in H^1(\Omega) \, : \, u \geq \varphi \text{ in } \Omega \text{ and }u=g \text{ on } \partial\Omega\right\},
\end{equation}
and the obstacle $\varphi$ satisfies $\varphi\leq g$ on $\partial\Omega$. 
We recall that the solution $u$ of the obstacle problem \eqref{eq:obP} satisfies the boundary value problem
\begin{equation}\label{eq:VI}
\left\{\begin{aligned}
\max\{-\nabla E[u],\varphi-u\}&= 0&&\text{in }\Omega\\
u &=g&&\text{ on }\partial\Omega.
\end{aligned}\right.
\end{equation}
This is a classical fact; we sketch the formal argument for completeness. If $v\in C^\infty(\bar{\Omega})$ is \emph{nonnegative}, then for any $\eps \geq 0$ we have $u + \eps v\in \A$ and hence
\[E[u + \eps v] -E[u] \geq 0.\]
Dividing by $\eps$ and sending $\eps\to 0^+$ yields
\begin{equation}\label{eq:VIt}
\int_\Omega \nabla E[u] v\, dx = \frac{d}{d\eps}\Big\vert_{\eps=0}E[u + \eps v]  \geq 0
\end{equation}
for all nonnegative $v$. Hence $\nabla E[u]\geq 0$ in $\Omega$. Furthermore, on any ball $B(x,r)\subset \Omega$ where $u > \varphi$, we can relax the nonnegativity constraint on $v$ and still ensure $u+\eps v\in \A$ for small $\eps>0$. It follows that $\nabla E[u] = 0$ on the set $\{u > \varphi\}$, which establishes \eqref{eq:VI}. We note that solutions of \eqref{eq:VI} are properly interpreted in the viscosity sense \cite{bardi2008optimal,crandall1992user}.

\subsection{PDE acceleration}
\label{sec:pdeacc}

We now show how to apply PDE acceleration to the obstacle problem \eqref{eq:obP}. For the moment, we consider the $L^2$-penalized formulation
\begin{equation}\label{eq:L2P}
\min_{u\in H^1(\Omega)}\left\{ \int_\Omega \Phi(x,\nabla u) + \frac{\mu}{2}(\varphi-u)_+^2\, dx \, : \,  u=g \text{ on } \partial\Omega\right\}.
\end{equation}
Theorem \ref{thm:rate} guarantees that PDE acceleration will converge with a linear rate for any finite $\mu>0$. However, we need to send $\mu\to \infty$ to recover the solution of the constrained problem \eqref{eq:obP}. We will see, however, that the accelerated PDE method for \eqref{eq:L2P} is insensitive to the choice of $\mu$, and can be easily solved for $\mu>0$ large, and in numerics we send $\mu\to \infty$ and obtain a very simple scheme for solving \eqref{eq:obP}. We explain in more detail below.

The PDE accelerated equations of motion \eqref{eq:motion} for the penalized problem \eqref{eq:L2P}  are 
\begin{equation}\label{eq:L2Pm}
u_{tt}+ au_t = -\nabla E[u] + \mu (\varphi-u)_+,
\end{equation}
subject to the Dirichlet condition $u=g$ on $\partial\Omega$, where
\[\nabla E[u] =\text{div}\left( \nabla_p\Phi(x,\nabla u) \right).\]
 We now discretize in time using the standard finite differences
\[u_t \approx \frac{u^{n+1}-u^n}{dt}, \ \text{ and } \ u_{tt}\approx\frac{u^{n+1}-2u^n+u^{n-1}}{dt^2}. \]
The important point now is that we handle the penalty term \emph{implicitly}. The discrete in time scheme becomes
\begin{equation}\label{eq:L2Ps}
(1+adt)u^{n+1}-\mu dt^2(\varphi-u^{n+1})_+ = (2+adt)u^n - u^{n-1} - dt^2\nabla E[u].
\end{equation}
Since the left hand side is strictly increasing in $u^{n+1}$, there is a unique solution of \eqref{eq:L2Ps}. We can compute the solution explicitly as follows:
\begin{equation}\label{eq:scheme1}
\left\{\begin{aligned}
v &=\frac{(2+adt)u^n - u^{n-1} - dt^2\nabla E[u]}{1+adt}  \\
w &=\frac{(2+adt)u^n - u^{n-1} - dt^2\nabla E[u] + \mu dt^2\varphi}{1+adt + \mu dt^2}  \\ 
u^{n+1}(x) &= \begin{cases}
v(x),&\text{if } v(x) \geq \varphi(x)\\
w(x),&\text{otherwise.}\end{cases}
\end{aligned}\right.
\end{equation}
The scheme is simple to implement, and the CFL condition is dictated solely by the discretization of $\nabla E[u]$ and is independent of the penalty $\mu$. In practice, we find the algorithm is completely insensitive to the choice of $\mu$ and runs efficiently for, say, $\mu > 10^{10}$. 

Instead of choosing a very large value for $\mu$, we can in fact send $\mu \to \infty$ in the scheme \eqref{eq:scheme1}. Indeed, the only place $\mu$ appears is in the update for $w$, and taking the limit as $\mu\to \infty$ we find that $w=\varphi$. Hence, we obtain the simpler scheme
\begin{equation}\label{eq:scheme2}
\boxed{\left\{\begin{aligned}
v &=\frac{(2+adt)u^n - u^{n-1} - dt^2\nabla E[u]}{1+adt}  \\
u^{n+1}(x) &= \max\{v(x),\varphi(x)\}
\end{aligned}\right.}
\end{equation}
as the limit of \eqref{eq:scheme1} as $\mu\to\infty$. In our simulations, we use the scheme \eqref{eq:scheme2}, since it is simpler and more intuitive, but the results are identical, up to machine precision, using scheme \eqref{eq:scheme1} with $\mu=10^{10}$.  We use finite differences to discretize $\nabla E[u]$ in this paper---in particular we discretize the gradient and divergence separately, using forward differences for $\nabla u$ and backward differences for the divergence. We set the damping parameter to be the optimal value $a=2\pi$ from the linear analysis in Section \ref{sec:optdamp}. We run the iterations \eqref{eq:scheme2} until 
\begin{equation}\label{eq:residual}
|\max\{-\nabla E[u^n],\varphi-u^n\}| \leq dx\|\varphi\|_{L^\infty}
\end{equation}
at all grid points.    

We should note there is nothing specific about finite difference schemes in this accelerated framework; one could just as easily use finite elements, spectral methods, or any other numerical PDE method. Once a discretization is settled on, the time step restriction on $dt$ follows from the CFL condition, which is straightforward to derive (see Section \ref{sec:app}).

\subsection{Primal dual algorithms}
\label{sec:pdual}

Recently in \cite{zosso2017efficient}, a primal dual algorithm was proposed for obstacle problems, and it was shown to be several orders of magnitude faster than existing state of the art methods. We compare PDE acceleration against an improved version of the primal dual algorithm from \cite{zosso2017efficient}, which is described below.

The primal dual algorithm solves the minimal surface obstacle problem
\begin{equation}\label{eq:obPM}
\min_{u\in \A} \int_\Omega \sqrt{1+|\nabla u|^2}\, dx,
\end{equation}
following roughly the outline in Section \ref{sec:primaldual}. We compute the convex dual of $\Phi(x) = \sqrt{1+|x|^2}$ to be
\begin{equation}\label{eq:cdual}
\Phi^*(p) = 
\begin{cases}
-\sqrt{1-|p|^2},&\text{if }|p|\leq 1\\
\infty,&\text{if }|p|> 1.
\end{cases}
\end{equation}
The primal dual algorithm from \cite{zosso2017efficient} for solving \eqref{eq:obPM} solves the equivalent primal dual formulation
\[\min_{u\geq \phi}\max_{|p|\leq 1}\int_{\Omega}p\cdot \nabla u + \sqrt{1-|p|^2}\, dx\]
by alternatively updating the primal variable $u$ and the dual variable $p$ with proximal updates. The full algorithm is given below.
\begin{equation}\label{eq:PD}
\left\{\begin{aligned}
p^{n+1}(x) &= \argmin_{|p|\leq 1} \left\{ -\nabla \bar{u}^n(x)\cdot p - \sqrt{1-|p|^2} + \frac{1}{2r_1}|p-p^{n}(x)|^2\right\}\\
u^{n+1} &= \max\{\phi,u^n + r_2\div(p^{n+1})\}\\
\bar{u}^{n+1} &= 2u^{n+1}-u^n.
\end{aligned}\right.
\end{equation}
The final step is an overrelaxation, and we set the Dirichlet condition $u=g$ on $\partial \Omega$ at each step. If the problem is discretized on a grid with spacing $dx$, then the method converges for any choices of $r_1,r_2$ with $r_1r_2\leq dx^2/6$ \cite{zosso2017efficient}. 
In fact, as noticed in Section \ref{sec:primaldual}, the ratio $r_1/r_2$ plays the role of the damping parameter $a$ in PDE acceleration \eqref{eq:scheme2}, allowing us to set $r_1/r_2=4\pi^2$, which is optimal for $\Omega=[0,1]^n$ via the linear analysis in Section \ref{sec:optdamp}.

While the update in the dual variable $p(x)$ is pointwise, it is not an explicit update and involves solving a constrained convex optimization problem. We contrast this with the PDE acceleration update \eqref{eq:scheme2} which is simple and explicit. 
In \cite{zosso2017efficient} the authors propose to solve the dual problem with iteratively re-weighted least squares (IRLS), that is, they define $\psi^0=p^n(x)$ and iteratively solve
\begin{equation}\label{eq:IRLS}
\psi^{k+1}=\argmin_{|\psi|\leq 1}\left\{-\nabla \bar{u}^n(x)\cdot \psi + \frac{1}{2}\frac{|\psi|^2 - 1}{\sqrt{1-|\psi^{k}|^2}} + \frac{1}{2r_1}|\psi-p^n(x)|^2\right\},
\end{equation}
setting $p^{n+1}(x) =\lim_{k\to\infty}\psi^k$. Actually, in \cite{zosso2017efficient} the factor of $\frac{1}{2}$ in front of the term $\frac{|\psi|^2 - 1}{\sqrt{1-|\psi^{k}|^2}}$ is missing; this is required to ensure that if the iterations converge, then the fixed point satisfies the correct optimality conditions for the original dual problem. It is claimed in \cite{zosso2017efficient,zosso2016minimal} that the IRLS iterations converge for $r_1$ sufficiently small. However, inspecting the proof in \cite[Lemma 4.4]{zosso2016minimal} it appears the restriction on $r_1$ is impractical for $dx\ll 1$. In practice, we find that for $dx \ll 1$ the IRLS iterations drift outside of the unit ball $|\psi^k|\leq 1$ after only a few iterations, in which case \eqref{eq:IRLS} is not well-defined and the iterations cannot continue. Simple fixes that we tried, such as projecting back onto the unit ball, were found to not be useful. We note we observed failure of the IRLS iterations even for small values of $r_1$.\footnote{Even if the IRLS iterations were to converge for extremely small $r_1>0$, the performance of the primal dual method is highly sensitive to the ratio $r_1/r_2$, and convergence of the primal dual iterations is extremely slow for very small or very large $r_1$.}

We propose another method for solving the dual problem that is robust and works for any value of $r_1$ and $dx$. We describe our method below. For convenience, let us define
\[F(p)=-\nabla \bar{u}(x)\cdot p - \sqrt{1-|p|^2} + \frac{1}{2r_1}|p-p^{n}(x)|^2.\]
We also compute
\begin{equation}\label{eq:Fg}
\nabla F(p) = -\frac{1}{r_1}(p^n(x) + r_1\nabla \bar{u}(x)) + \frac{p}{\sqrt{1-|p|^2}} + \frac{1}{r_1}p.
\end{equation}
Then the dual problem is $p^{n+1}(x)=\argmin_{|p|\leq 1}F(p)$. We first note that since $F( (1-\eps)p)< F(p)$ for any $p$ with $|p|=1$ and $\eps>0$ sufficiently small, we must have $|p^{n+1}(x)|<1$, and so $\nabla F(p^{n+1}(x))=0$. For any $\eta$ with $\eta\cdot p^{n+1}(x) = 0$ we have
\[0 =\eta\cdot r_1\nabla F(p^{n+1}(x)) = -(p^n(x) + r_1\nabla \bar{u}(x))\cdot \eta.\]
Therefore, $p^{n+1}(x)=\alpha q(x)$ for some $\alpha\in (-1,1)$, where 
\begin{equation}\label{eq:qofx}
q(x)=
\begin{cases}
\frac{p^n(x) + r_1\nabla \bar{u}(x)}{|p^n(x) + r_1\nabla \bar{u}(x)|},&\text{if } p^n(x) + r_1\nabla \bar{u}(x) \neq 0,\\
0,&\text{otherwise.}
\end{cases}
\end{equation}
The value of $\alpha\in (-1,1)$ is the unique root of the function
\begin{equation}\label{eq:fbis}
f(\alpha):=q(x)\cdot  r_1\nabla F(\alpha q(x)) = \alpha + \frac{r_1\alpha}{\sqrt{1-\alpha^2}}-|p^n + r_1\nabla \bar{u}(x)|.
\end{equation}
Since $F$ is strictly convex, $f$ is strictly increasing in $\alpha$, and so we can compute the root of $f$ with a simple bisection search. Inspecting \eqref{eq:fbis} we see that $\alpha\in [0,\min\{1,N\}]$, where $N=|p^n + r_1\nabla \bar{u}(x)|$. For $\alpha$ is this range, we can perform some algebraic manipulations on $f$ to see that we can instead bisect on the function 
\[g(\alpha) = r_1^2 \alpha^2 - (1-\alpha^2)(\alpha-N)^2,\]
which does not involve the costly square root operation. The method is guaranteed to converge, and the accuracy is directly related to the number of bisection iterations, that is
\[\text{Bisection Search Error} \leq \frac{1}{2^{k+1}},\]
where $k$ is the number of bisections.

We emphasize that the IRLS method proposed in \cite{zosso2017efficient} does not converge for any of the simulations presented in Section \ref{sec:app}. Thus, the new bisection method is required to allow comparisons against the primal dual algorithm for the nonlinear minimal surface problem. 

In our implementation of the primal dual method, we use forward differences for $\nabla u$ and backward differences for the divergence, as in \cite{zosso2017efficient}. We set the ratio $r_1/r_2=4\pi^2$ based on an optimal linear analysis as in Section \ref{sec:optdamp}, along with the CFL condition $r_1r_2 \leq dx^2/6$ \cite{zosso2017efficient}. We choose the number of bisection iterations so that the dual problem is solved to an accuracy of $\eps dx^2$, where $\eps$ is the accuracy to which we wish to solve the obstacle problem. This requires around 30 iterations for most of our simulations. We run the algorithm until the residual condition \eqref{eq:residual} is satisfied.

\section{Experiments}
\label{sec:app}

We give here some applications of the PDE acceleration method for solving various obstacle problems. All algorithms, including our improved primal dual method and the $L^1$-penalty method \cite{tran20151}, were implemented in C and run on a laptop with a 64-bit 2.20GHz CPU. The code for all simulations is available on GitHub: \url{https://github.com/jwcalder/MinimalSurfaces}.

\subsection{Minimal surface obstacle problems}

We first consider the constrained minimal surface problem
\begin{equation}\label{eq:cms}
\min \left\{ \int_{\Omega} \sqrt{1 + |\nabla u|^2} \,dx \ : \ u\in H^1_0(\Omega) \text{ and }u\geq \phi\right\}.
\end{equation} 
Here, $\phi:\Omega\to \R$ is the obstacle and $\Omega=[0,1]^2$. We solve the problem with the PDE acceleration scheme \eqref{eq:scheme2} using the implementation described in Section \ref{sec:pdeacc}. Here, 
\[\nabla E[u] = -\text{div}\left( \frac{\nabla u}{\sqrt{1 + |\nabla u|^2}} \right),\]
and the CFL condition dictates that $dt < dx/\sqrt{2}$. We set $dt=0.8dx/\sqrt{2}$.

\begin{figure}
\centering
\subfloat[Obstacle $\varphi:=\varphi_1/100$]{\includegraphics[trim = 40 40 40 40, clip=true,width=0.50\textwidth]{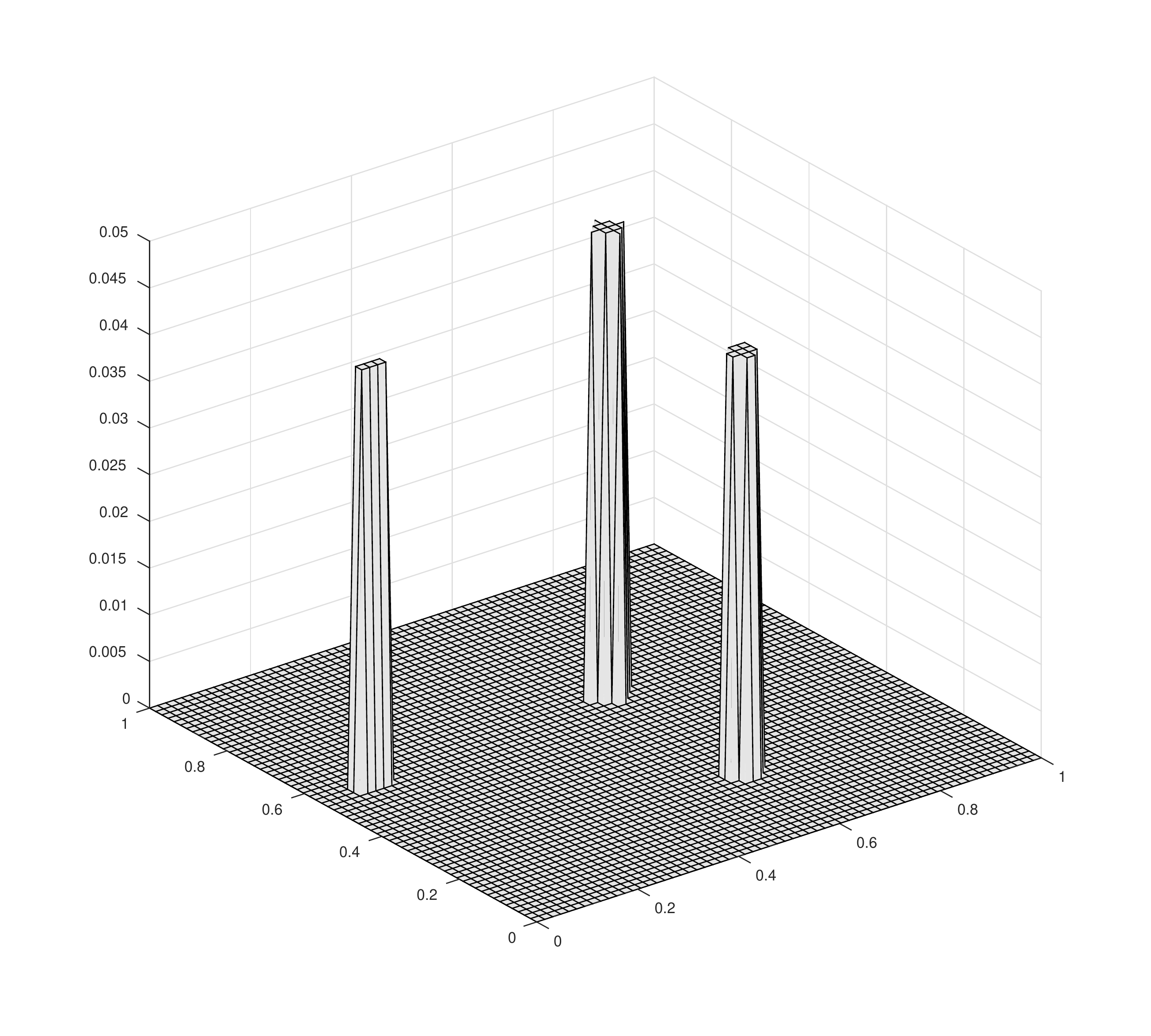}}
\subfloat[Minimal surface for $\varphi:=\varphi_1/100$]{\includegraphics[trim = 40 40 40 40, clip=true,width=0.50\textwidth]{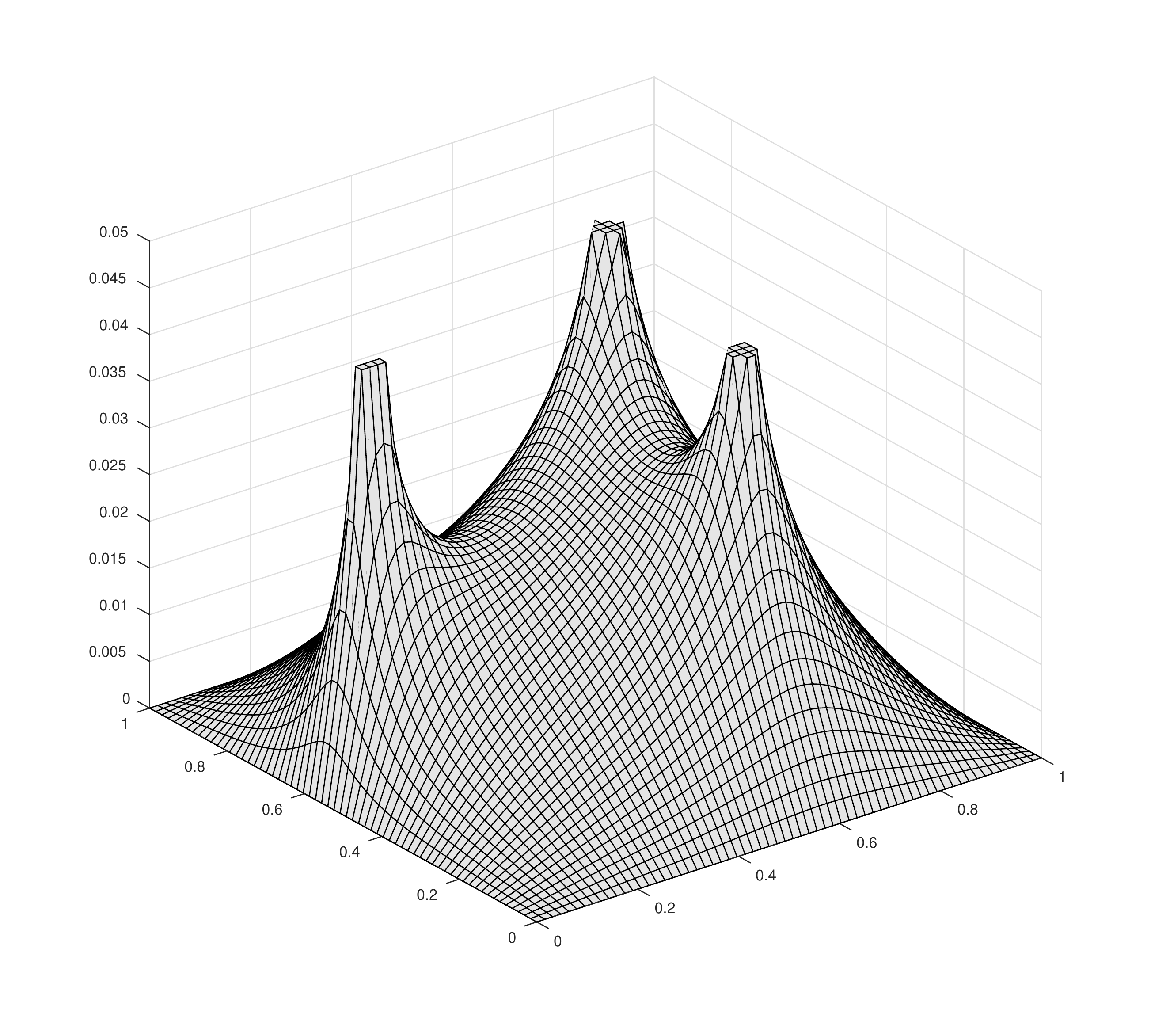}\label{fig:MS11}}

\subfloat[Minimal surface $\varphi:=\varphi_1/50$]{\includegraphics[trim = 40 40 40 40, clip=true,width=0.50\textwidth]{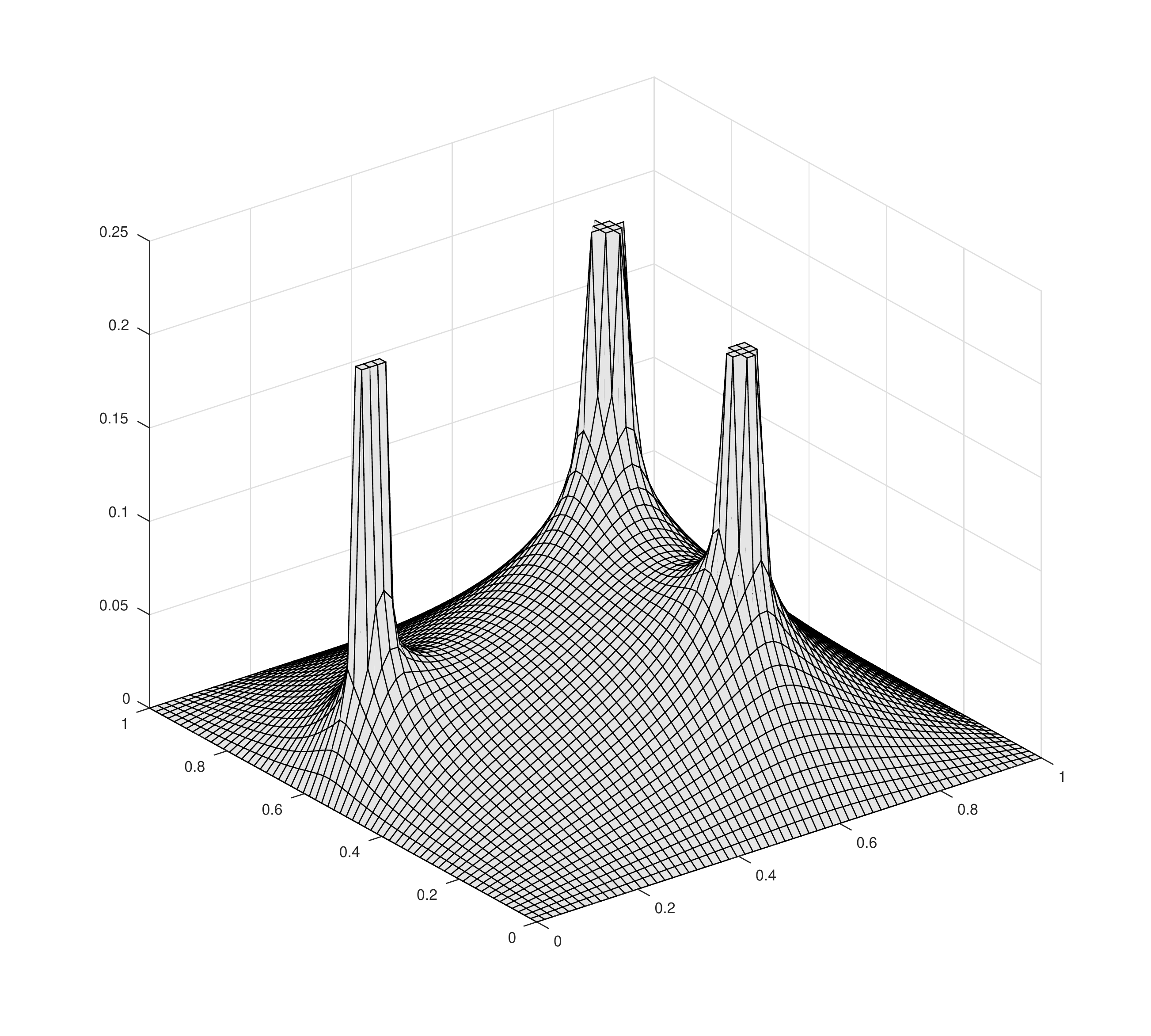}\label{fig:MS15}}
\subfloat[Minimal surface $\varphi:=\varphi_1/10$]{\includegraphics[trim = 40 40 40 40, clip=true,width=0.50\textwidth]{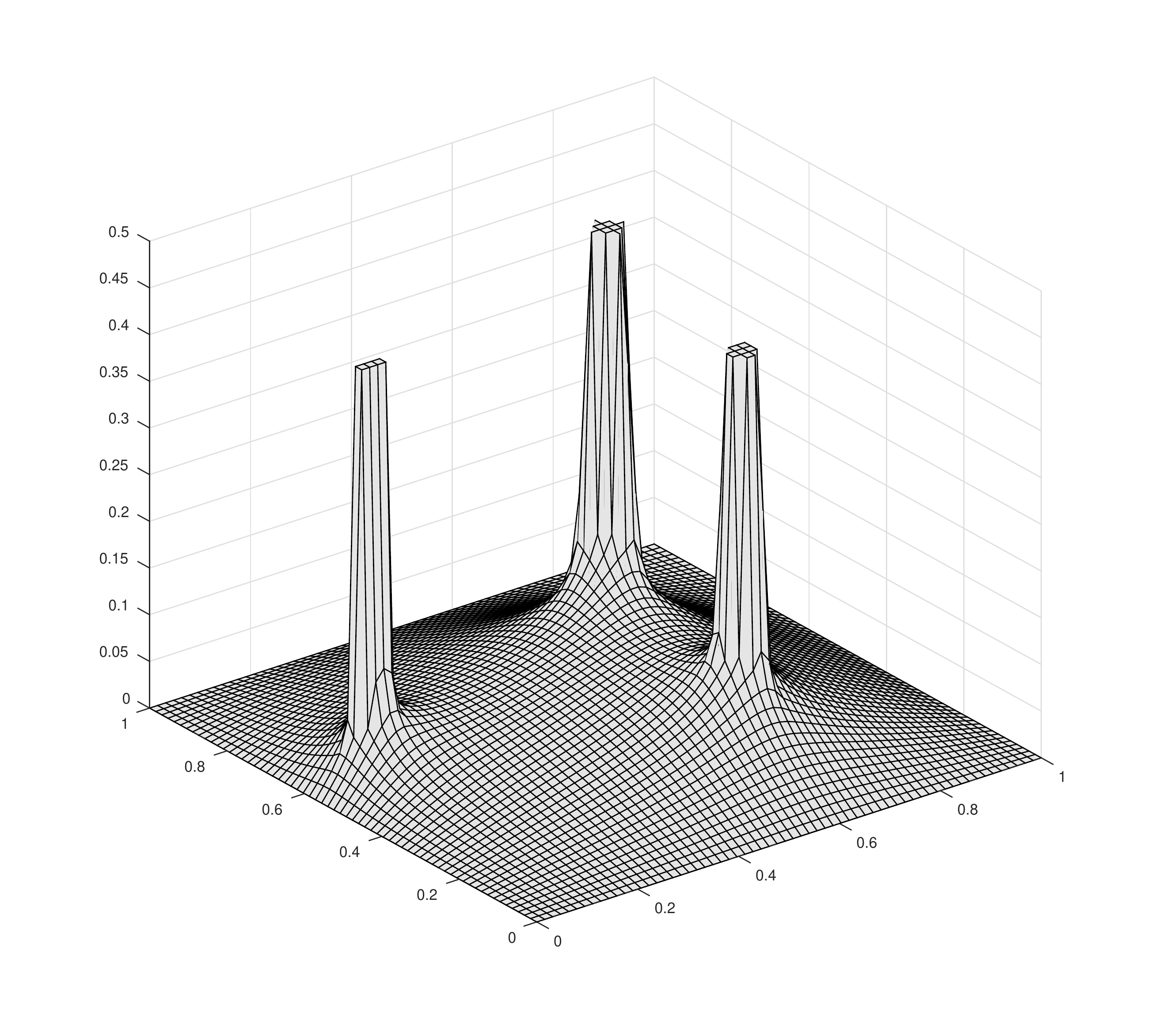}\label{fig:MS110}}
\caption{Minimal surfaces for obstacle $\varphi_1$ computed with PDE acceleration on a $64\times 64$ grid.}
\label{fig:MS1}
\end{figure}

The first obstacle we consider is from \cite{zosso2017efficient} and is given by
\begin{equation}\label{eq:o1}
\varphi_1(x_1,x_2)=
\begin{cases}
5&\text{for }|x_1-0.6|+|x_2-0.6| < 0.04\\
4.5&\text{for }(x_1-0.6)^2 + (x_2-0.25)^2 < 0.001\\
4.5&\text{for }x_2=0.57 \text{ and } 0.075 < x_1 < 0.13\\
0&\text{otherwise.}
\end{cases}
\end{equation}
Figure \ref{fig:MS1} shows the obstacle $\phi:=\phi_1/100$, and the minimal surfaces computed with the PDE acceleration algorithm for  $\phi:=\phi_1/100$, $\phi:=\phi_1/20$, and $\phi:=\phi_1/10$.  Figure \ref{fig:MS11} shows a short obstacle with small deflections, and the solution in this case is well-approximated by the linearized minimal surface problem
\begin{equation}\label{eq:cmsl}
\min \left\{ \int_{\Omega} \frac{1}{2}|\nabla u|^2\,dx \ : \ u\in H^1_0(\Omega) \text{ and }u\geq \phi\right\}.
\end{equation} 
The obstacles in Figures \ref{fig:MS15} and \ref{fig:MS110} are significantly taller and the true minimal surfaces are not well-captured by linearization. In particular, the minimal surface for $\varphi:=\varphi_1$ is nearly identically zero, which we show in Figure \ref{t}.

We remark that this is in contrast with previous work (see \cite[Fig.~3(d)]{zosso2017efficient}), which reported that the minimal surface for $\phi_1$ resembles Figure \ref{fig:MS11} (the minimal surface for $\phi_1/100$).  We show the true minimal surface for $\phi_1$, computed by our algorithm, in Figure \ref{t}, and the solution of the linearized minimal surface equation in Figure \ref{l}, which agrees very closely by eye with \cite[Fig.~3(d)]{zosso2017efficient}, suggesting that \cite{zosso2017efficient} are in fact solving the linearized minimal surface problem, and not the true nonlinear minimal surface problem. We suspect this is due to the authors of \cite{zosso2017efficient} mistakenly taking $dx=1$ in their code, which has the effect of drastically reducing the height of the obstacles and putting one in the linear setting. Since the nonlinear minimal surface equation is not homogeneous in the gradient, one cannot scale away the spatial resolution as can be done for the linearized equation.
\begin{figure}
\centering
\subfloat[True minimal surface]{\includegraphics[clip=true, trim = 30 20 20 20, width=0.5\textwidth]{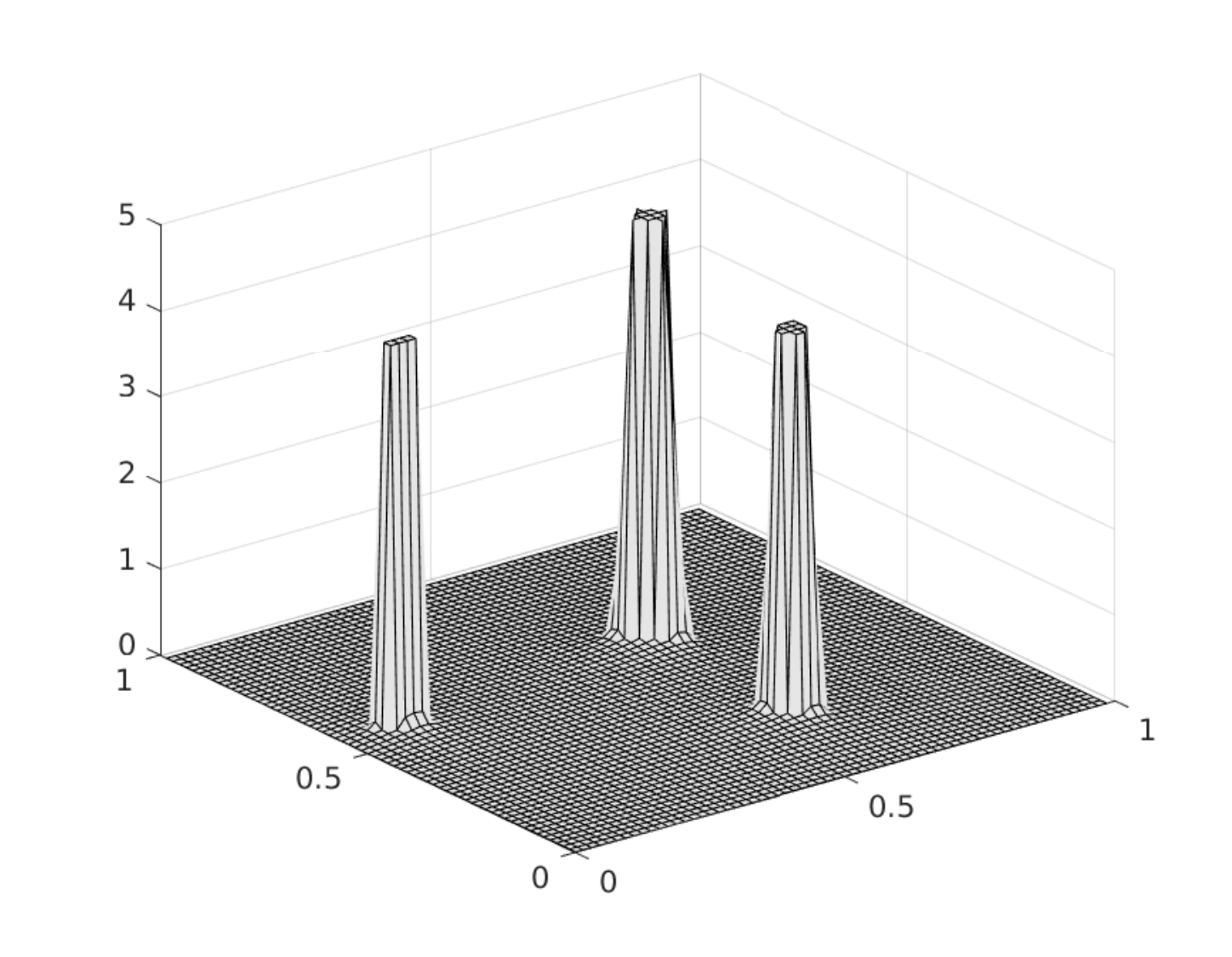}\label{t}}
\subfloat[Solution of linearized problem]{\includegraphics[clip=true, trim = 30 20 20 20, width=0.5\textwidth]{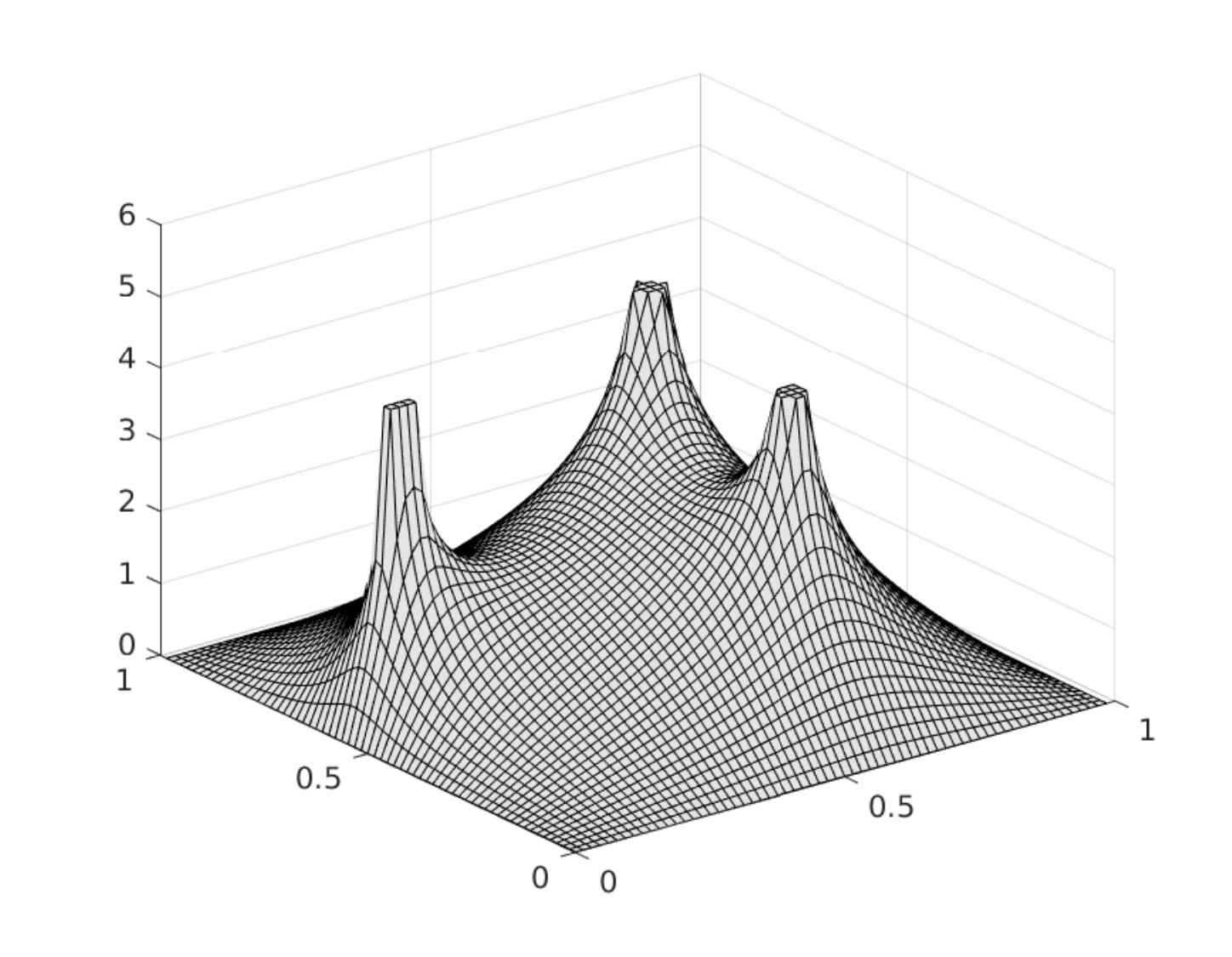}\label{l}}
\caption{Comparison of (a) the true minimal surface for $\phi_1$, and (b) the solution of the linearized problem, which is presented incorrectly in \cite[Figure 3d]{zosso2017efficient} as the true solution of the nonlinear minimal surface problem on a $64\times 64$ grid.}
\label{fig:MS}
\end{figure}

It is easy to see that Figure \ref{l} cannot be the true minimal surface for $\phi_1$ by computing the surface area of the solutions: Our solution in Figure \ref{t} has surface area $3.9855$, while the surface area of Figure \ref{l} is $8.5105$. We can also compute by hand the surface area of the obstacle $u\equiv \phi_1$. The obstacles are a square, circle, and a line segment. The square has side lengths $\ell = \sqrt{0.04^2 + 0.04^2} = 0.0566$, so perimeter is $0.2263$ and area $A_1:=0.0032$. The height of the square is $5$, so the surface area contributed by the obstacle would be $S_1:=1.1315$. The circle has radius $r=\sqrt{0.001} \approx 0.0316$, and so perimeter is $2\pi r \approx 0.1987$ and area is $A_2:=0.0031$. The height of the circle is $4.5$, so it contributes surface area $S_2:=0.8941$. Finally, the line segment has width $0.055$ and height $4.5$, so the surface area contribution is $S_3:=0.2475$, while the line segment contributes zero area in the projection to the plane. The surface area of the solution is then
\[\text{Surface Area } = 1 - A_1 - A_2 + S_1 + S_2 + S_3 = 3.2668,\]
which agrees more closely with our solution, up to discretization errors on the coarse $64\times 64$ grid.

\begin{table}[!t]
\centering
\begin{tabular}{|c|c|c|c|c|c|c|}
 \hline
 &\multicolumn{6}{c|}{Obstacle $\phi:=\phi_1/50$}\\
 \hline
 &\multicolumn{2}{c|}{\textbf{Our Method}}&\multicolumn{2}{c|}{\textbf{Primal Dual} \cite{zosso2017efficient}}&\multicolumn{2}{c|}{\textbf{$L^1$ penalty} \cite{tran20151}}\\
\hline
Mesh & Time & Iter.& Time & Iter. &Time & Inner (outer) iter.  \\
\hline
$64\times 64$     & 0.023 	& 360   & 0.186 	& 370    & 0.4 & 7254 (2380)    \\ 
$128\times 128$   & 0.144 	& 823   & 1.76 	& 870    & 5.6 & 29284 (10340) \\ 
$256\times 256$   & 1.22 	& 1863  & 18 	   & 2070   &   &        \\ 
$512\times 512$   & 12 	   & 4135  & 163 	   & 4390   &      &        \\ 
$1024\times 1024$ & 107 	& 9074  & 1650 	& 10210  &      &        \\
\hline
Complexity  & 1.54& 0.58   & 1.64     & 0.60 &  1.9  & 0.93       \\
\hline
\end{tabular}
\caption{ Run times in seconds and number of iterations for the PDE accelerated solver, primal dual method \cite{zosso2017efficient}, and $L^1$ penalty method \cite{tran20151} for solving the \emph{nonlinear} minimal surface obstacle problem with obstacle $\phi_1$. The complexity is measured as a function of the number of grid points used in the discretization. }
\label{tab:MS1sima}
\end{table}

\begin{table}[!t]
\centering
\begin{tabular}{|c|c|c|c|c|c|c|}
 \hline
 &\multicolumn{6}{c|}{Obstacle $\phi:=\phi_1/50$}\\
 \hline
 &\multicolumn{2}{c|}{\textbf{Our Method}}&\multicolumn{2}{c|}{\textbf{Primal Dual} \cite{zosso2017efficient}}&\multicolumn{2}{c|}{\textbf{$L^1$ penalty} \cite{tran20151}}\\
\hline
Mesh & Time & Iter.& Time & Iter. &Time & Inner (outer) iter.  \\
\hline
$64\times 64$    & 0.014 & 288    & 0.23     & 294   & 0.23 & 4969 (206)  \\ 
$128\times 128$   & 0.1   & 618    & 2.4      & 684   & 2.94 & 18228 (385) \\ 
$256\times 256$   & 0.9 	& 1323   & 23.6     & 1556  & 50.9 & 70177 (1221) \\ 
\hline
Complexity  & 1.5 & 0.55   & 1.67     & 0.6 & 1.95  & 0.96   (0.64)   \\
\hline
\end{tabular}
\caption{ Run times in the same setting as Table \ref{tab:MS1sima}, except with the relaxed stopping condition $\|u^{n+1}-u^n\|_{L^\infty}\leq dx^2/100$. In all other experiments we use the stopping condition \eqref{eq:residual}.  }
\label{tab:MS1sima_stop}
\end{table}

We now compare runtimes and iteration counts for our PDE acceleration method against the primal dual method \cite{zosso2017efficient} with our improved bisection method for solving the dual problem presented in Section \ref{sec:pdual}, and against the $L^1$-penalty method from \cite{tran20151}.  Table \ref{tab:MS1sima} compares the run times for the obstacle $\varphi:=\varphi_1/50$. Our method is more than 10x faster than primal dual in terms of CPU time, while both algorithms have similar iteration counts. The difference is that the PDE acceleration updates are explicit, while the dual update for the primal dual algorithm is implicit and involves solving a nonlinear optimization problem. We note both algorithms converge to a surface that looks to the eye similar to the minimal surface in about half of the iterations reported, and the final iterations are used to resolve the accuracy to the desired tolerance. For the $L^1$-penalty method, we used parameters $\lambda=100$, $\mu=500$, $L=2/dx^2$, and $dt = 1/4L$, which gave the best performance over the parameters we tried, and we report both the inner and outer iteration counts for completeness. Each inner iteration has similar complexity to a PDE acceleration iteration. The $L^1$-penalty method did not converge to our strict stopping condition \eqref{eq:residual} for grids of size $256\times 256$ or larger (we stopped the experiment at 18 minutes and 1 million iterations on the $256\times 256$ grid).

Our stopping condition \eqref{eq:residual} is standard in rigorous scientific computing, and is simply asking that all methods solve the same problem to the same accuracy.  It is also common to use the difference between subsequent iterates as a stopping condition, however, this can have a different meaning for each algorithm and does not provide any guarantee that the algorithms are solving the correct problem (e.g., one can reduce the time step to encourage ``faster'' convergence). To be complete, we have included in Table \ref{tab:MS1sima_stop} a comparison of PDE acceleration, primal dual, and the $L^1$-penalty method for the stopping condition
\begin{equation}\label{eq:stopping2}
\|u^{n+1}-u^n\|_{L^\infty}\leq C dx^2,
\end{equation}
for a constant $C$, which is the same condition used in \cite{tran20151}. We took $C=1/100$, since for any larger value of $C$, the $L^1$-penalty method stopped at a surface that was clearly by eye very far from the true minimal surface. We only computed the table up to a $256\times 256$ grid, since the $L^1$-penalty method did not converge in a reasonable amount of time for larger grids. We see that PDE acceleration and primal dual have similar complexity, though PDE acceleration is roughly $10$-times faster, and the $L^1$-penalty method is an order of magnitude slower.

Table \ref{tab:MS1sima} (and all future tables) also show computational complexity, which is computed as the exponent $p>0$ for which the curve $N^p$ most closely fits the CPU time or iteration count, where $N$ is the number of grid points used in the discretization. We see the complexity of PDE acceleration is $p\approx 1.55$, which agrees with the discussion in Section \ref{sec:complexity} for 2D problems. The primal dual method has slightly worse complexity ($p\approx 1.65$), and the $L^1$ penalty method has complexity $p\approx 2$. We note that complexity for the primal dual and $L^1$-penalty methods are different than those reported in \cite{zosso2017efficient} and \cite{tran20151}, respectively. The reason for this is that \cite{zosso2017efficient} and \cite{tran20151} report complexity for the \emph{linearized} minimal surface problem, and in particular, do not report runtimes or complexity for nonlinear problems. Furthermore, \cite{tran20151} reports \emph{sublinear} complexity $N^{0.85}$, which is either a numerical artifact or the result of stopping the iterations too early by choosing $C$ too large in \eqref{eq:stopping2}. Indeed, as soon as one visits each grid point \emph{once}, the complexity must be at least linear (e.g., $N^1$).

Let us explain briefly why the nonlinear minimal surface problem is more computationally complex to solve via $L^1$-penalty and primal dual methods. The $L^1$-penalty method \cite{tran20151} for linear problems involves solving a linear Poisson equation at each outer iteration, which can be done in linear time with multigrid methods, for example, while for nonlinear problems the outer iteration involves solving a nonlinear minimal surface problem, which is more expensive. Interestingly, the authors of \cite{tran20151} use Nesterov acceleration to solve the nonlinear problem at each outer iteration. However, the use of Nesterov acceleration does not employ our optimal damping parameter from Section \ref{sec:DP}, and the momentum is reset at each outer iteration, which we find inhibits the acceleration obtained from momentum methods. Regarding the primal dual method \cite{zosso2017efficient}, when solving the linearized problem one can use a handful of IRLS steps to solve the dual problem, as described in \cite{zosso2017efficient}. However, for nonlinear problems the IRLS method fails to converge and our new bisection method (see Section \ref{sec:pdual}) is required, which requires approximately $30-50$ iterations. This makes the dual problem more expensive to solve for nonlinear minimal surface problems and explains the difference between computation time for linear and nonlinear problems. To be clear, the IRLS method is unusable for nonlinear problems, since it returns complex numbers after a few iterations, so it does not even provide an approximate solution to the problem.

\begin{table}[!t]
\centering
\begin{tabular}{|c|c|c|c|c|c|c|}
 \hline
 &\multicolumn{6}{c|}{Obstacle $\phi:=\phi_2$}\\
 \hline
 &\multicolumn{2}{c|}{\textbf{Our Method}}&\multicolumn{2}{c|}{\textbf{Primal Dual} \cite{zosso2017efficient}}&\multicolumn{2}{c|}{\textbf{$L^1$ penalty} \cite{tran20151}}\\
\hline
Mesh & Time & Iter.& Time & Iter. &Time & Inner (outer) iter.  \\
\hline
$64\times 64$     & 0.012 	& 300   & 0.182 	& 330 & 0.31   & 7065 (60)   \\ 
$128\times 128$   & 0.138 	& 704   & 1.82 	& 780 & 3.4    & 19712 (70)   \\ 
$256\times 256$   & 1.08 	& 1620  & 17.8 	& 1720& 39.8   & 58788 (170) \\ 
$512\times 512$   & 10.2 	& 3642  & 180 	   & 4320& 551.1  & 199323 (470) \\ 
$1024\times 1024$ & 95.1 	& 8117  & 1880 	& 9710& 8401  &  660908 (1030) \\
\hline
Complexity  & 1.61& 0.59   & 1.66     & 0.61 &  1.84  & 0.82 (0.55)     \\
\hline
\end{tabular}
\caption{ Run times in seconds and number of iterations for the PDE accelerated solver, primal dual method \cite{zosso2017efficient}, and $L^1$ penalty method \cite{tran20151} for solving the \emph{nonlinear} minimal surface obstacle problem with obstacle $\phi_2$. }
\label{tab:MS1simb}
\end{table}

The second obstacle we consider is
\begin{equation}\label{eq:o2}
\varphi_2(x) = \sqrt{(1-|x-P|^2/0.09)_+}+\sqrt{(1-|x-Q|^2/0.0025)_+},
\end{equation}
where $P=(0.55,0.5)$ and $Q=(0.1,0.5)$. Figure \ref{fig:MS2} shows the obstacle and minimal surface. The run times for the PDE acceleration, primal dual, and the $L^1$-penalty method are shown in Table \ref{tab:MS1simb}. We again see that PDE acceleration is approximately 10x faster in terms of CPU time, and the $L^1$-penalty method is an order of magnitude slower.

\begin{figure}
\centering
\subfloat[Obstacle $\varphi_2$]{\includegraphics[trim = 40 40 40 40, clip=true,width=0.50\textwidth]{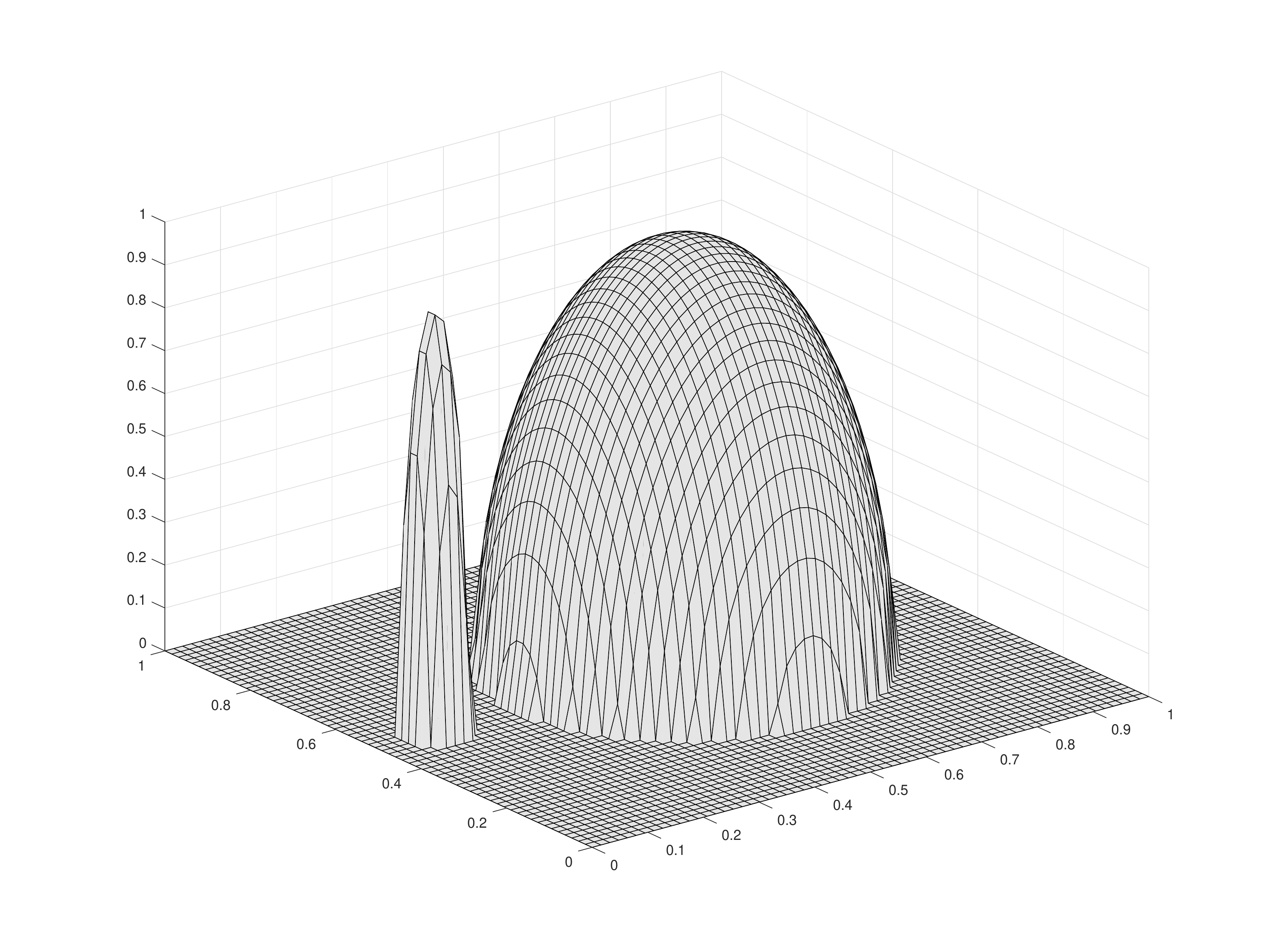}}
\subfloat[Minimal surface for $\varphi_2$]{\includegraphics[trim = 40 40 40 40, clip=true,width=0.50\textwidth]{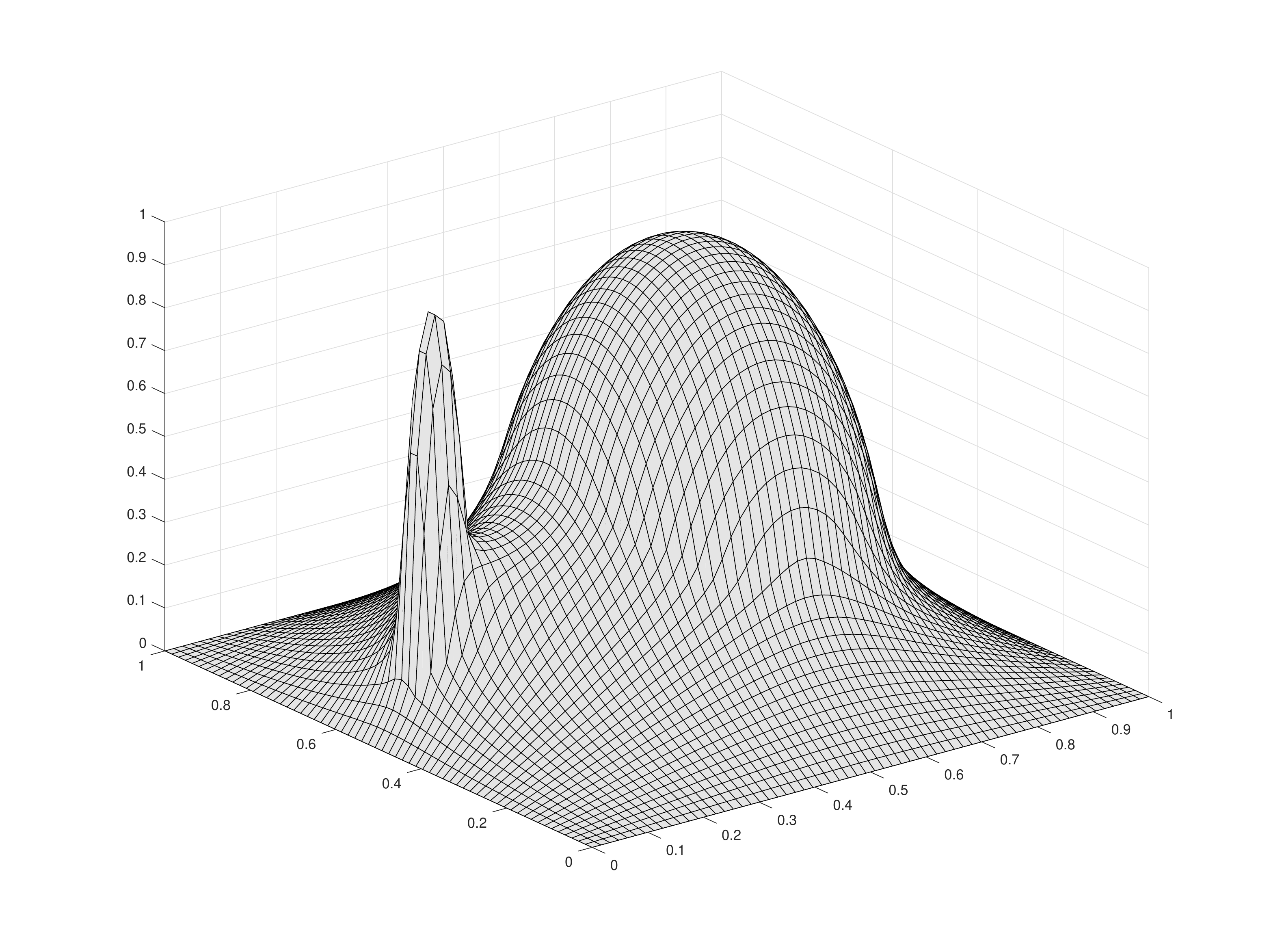}\label{fig:MS21}}

\caption{Minimal surface for the obstacle $\varphi_2$ computed with PDE acceleration on a $64\times 64$ grid.}
\label{fig:MS2}
\end{figure}

\subsection{Double obstacle with forcing}

Here, we consider the double obstacle problem with forcing from \cite{zosso2017efficient} (originally from \cite{wang2008algorithm};  see also \cite{badea2003convergence} for double obstacle problems ). The nonlinear version of the problem is given by
\begin{equation}\label{eq:cmsd}
\min \left\{ \int_{\Omega} \sqrt{1 + |\nabla u|^2} -uv\,dx \ : \ u\in H^1_0(\Omega) \text{ and }\psi \geq u\geq \phi\right\}.
\end{equation} 
The scheme \eqref{eq:scheme2} is simple to modify for the double obstacle problem by setting $u^{n+1}=\max\{\min\{v,\psi\},\phi\}$ at each iteration. We test PDE acceleration on the elasto-plastic torsion problem, originally from \cite{wang2008algorithm}. In this setting, $\Omega=[0,1]^2$, $\phi_3(x) = -\text{dist}(x,\partial \Omega)$, $\psi_3(x) = 0.2$, $u=0$ on $\partial\Omega$, and the force $v$ is given by
\begin{equation}\label{eq:force}
v(x) = 
\begin{cases}
300,&\text{if }x \in S:=\{|x_1-x_2|\leq 0.1 \text{ and } x_1\leq 0.3\}\\
-70e^{x_2}g(x),&\text{if }s\not\in S \text{ and } x_1 \leq 1-x_2\\
15e^{x_2}g(x),&\text{if }s\not\in S \text{ and } x_1 > 1-x_2,
\end{cases}
\end{equation}
where
\begin{equation}\label{eq:gf}
g(x) = 
\begin{cases}
6x_1,&\text{if }0 \leq  x_1 \leq  1/6\\
2(1-3x_1),&\text{if }1/6 <  x_1 \leq  1/3\\
6(x_1-1/3),&\text{if }1/3 <  x_1 \leq  1/2\\
2(1-3(x_1-1/3)),&\text{if }1/2 <  x_1 \leq  2/3\\
6(x_1-2/3),&\text{if }2/3 <  x_1 \leq  5/6\\
2(1-3(x_1-2/3)),&\text{if }5/6 <  x_1 \leq  1.
\end{cases}
\end{equation}
We then set $\phi:=\phi_3/10$, $\psi =\psi_3/10$ and $v:=v_3/10$ to get similar results to \cite{zosso2017efficient} where the linearization is studied. For comparison with \cite{zosso2017efficient} we also consider the linear double obstacle problem
\begin{equation}\label{eq:cmsdLin}
\min \left\{ \int_{\Omega} \frac{1}{2}|\nabla u|^2 -uv\,dx \ : \ u\in H^1_0(\Omega) \text{ and }\psi \geq u\geq \phi\right\}.
\end{equation} 

We report the CPU runtimes and iteration counts for the PDE acceleration method and the improved primal dual method for both the linear and nonlinear double obstacle problems in Table \ref{tab:DOsim}. We see that for the nonlinear problem, PDE acceleration is again roughly 10x faster than primal dual, while only 2x faster for the linear obstacle problem. The difference is that the dual update is explicit for linear problems, which leads to a substantial acceleration. Figure \ref{fig:MSDO} shows the computed membrane on a $64\times 64$ grid, and the double obstacle contact regions computed on a $512\times 512$ grid. These agree well with the results in \cite[Fig.~5]{zosso2017efficient}.

\begin{table}[!t]
\centering
\begin{tabular}{|c|c|c|c|c|c|c|c|c|}
\hline &\multicolumn{4}{c|}{Linear double obstacle problem}&\multicolumn{4}{c|}{Nonlinear double obstacle problem}\\
 \hline &\multicolumn{2}{c|}{\textbf{Our Method} }&\multicolumn{2}{c|}{\textbf{Primal Dual} \cite{zosso2017efficient}}&\multicolumn{2}{c|}{\textbf{Our Method} }&\multicolumn{2}{c|}{\textbf{Primal Dual} \cite{zosso2017efficient}}\\
\hline
Mesh & Time & Iter. & Time & Iter. &Time & Iter. &Time & Iter. \\
\hline
$64^2$    & 0.012	& 378  & 0.01 	& 356   & 0.016 	& 382  & 0.156 	& 360l      \\ 
$128^2$   & 0.1	& 835  & 0.086	& 814   & 0.133 	& 862  & 1.59 	   & 810        \\ 
$256^2$   & 0.69	& 1807 & 0.785	& 1884  & 1.23 	& 1937 & 15.2 	   & 1810          \\ 
$512^2$   & 5 	   & 3937 & 7.56 	& 4092  & 11.9 	& 4297 & 143 	   & 4050           \\ 
$1024^2$  & 62.9 	& 8459 & 81.7 	& 9113  & 108 	   & 9409 & 1.540 	& 9000            \\ 
\hline
Comp.& 1.52 & 0.56  & 1.62  & 0.58  &  1.6    & 0.58  & 1.65     &0.58     \\
\hline
\end{tabular}
\caption{ Run times in seconds for the PDE accelerated solver compared to the primal dual method from \cite{zosso2017efficient} for the \emph{linear} and \emph{nonlinear} minimal surface double obstacle problem with forcing (the elasto-plastic torsion problem). }
\label{tab:DOsim}
\end{table}

\begin{figure}
\centering
\subfloat[Membrane ($64\times 64$ grid)]{\includegraphics[trim = 40 40 40 40, clip=true,width=0.50\textwidth]{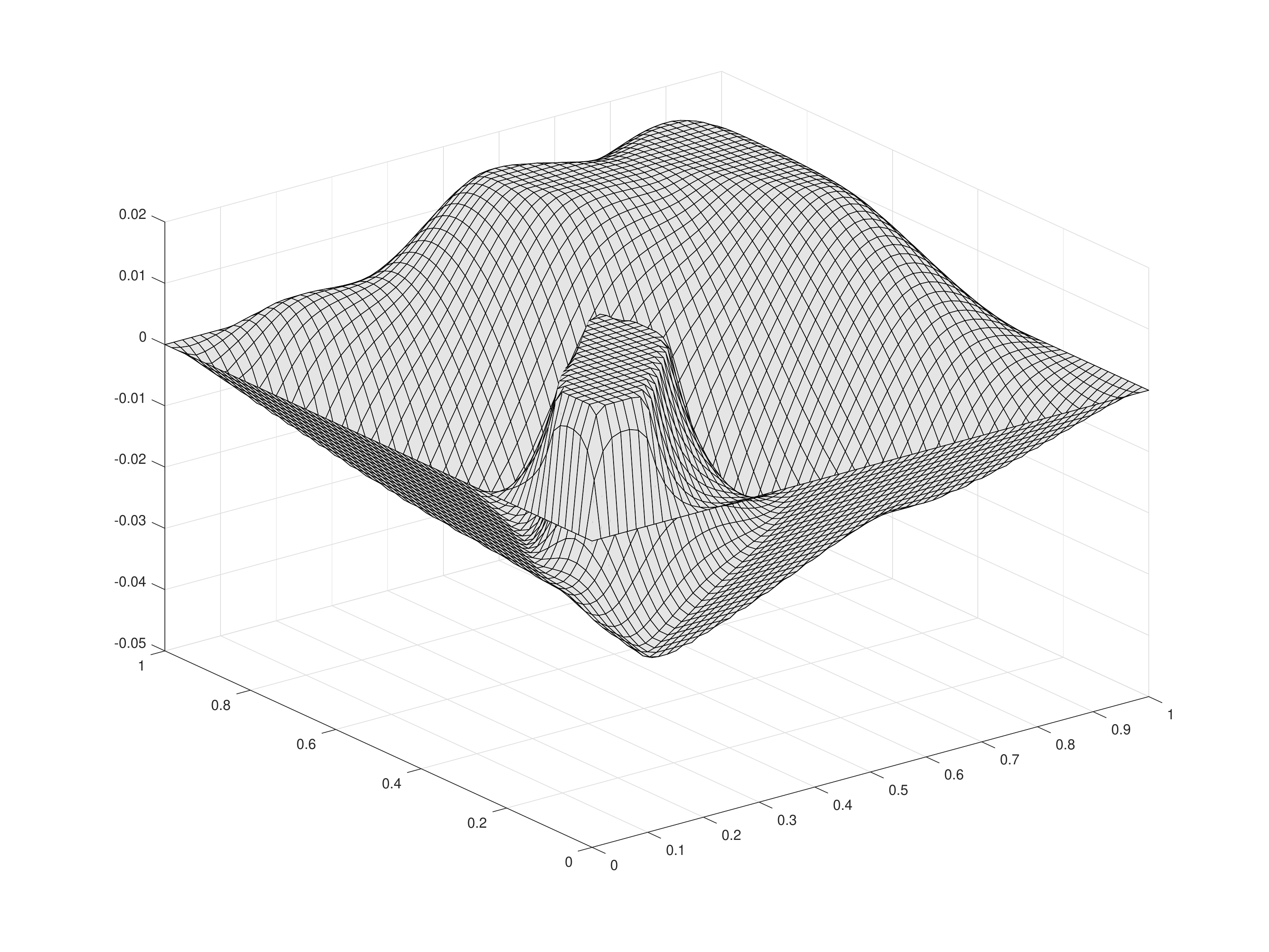}}
\subfloat[Double obstacle contact regions]{\includegraphics[width=0.50\textwidth]{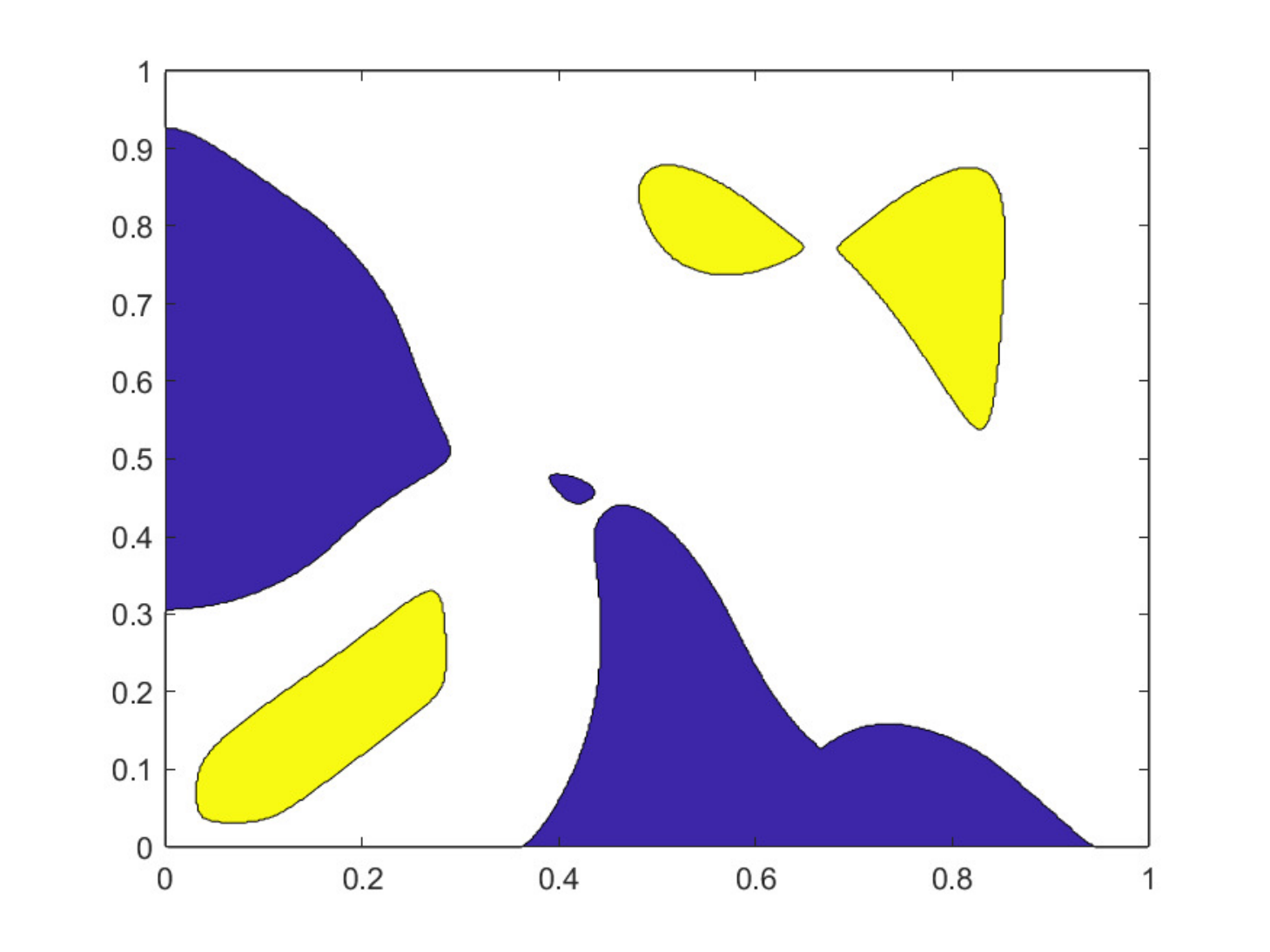}}
\caption{Depiction of the computed membrane and contact regions for the double obstacle with forcing example from \cite{zosso2017efficient}. The contact regions were computed on a $512\times 512$ grid.}
\label{fig:MSDO}
\end{figure}

\subsection{Stochastic homogenization with obstacles}

We consider the stochastic obstacle problem
\begin{equation}\label{eq:stob}
\min_{u\in H^1_0(\Omega)}\left\{ \int_\Omega \frac{1}{2}|A(\tfrac{x}{\eps})\nabla u|^2 - fu\,dx\, :\, u \geq \varphi \text{ in }\Omega  \right\},
\end{equation}
where $A(x)$ is sampled from a $\Z^d$-stationary probability measure with unit range dependence. We consider here a random checkerboard, where we let $(b(z))_{z\in \Z^2}$ be independent random variables such that 
\[\P(b(z)=1) = \P(b(z) = 9) = \frac{1}{2},\]
and set $A(x) = b(z)$ for $x\in z + [0,1)^2$.
We also set $f=1$. By the Dynkin formula~\cite[Ex.~2.3]{armstrong2017quantitative}, solutions of \eqref{eq:stob} converge almost surely to solutions of the homogenized problem
\begin{equation}\label{eq:stobhom}
\min_{u\in H^1_0(\Omega)}\left\{ \int_\Omega \frac{1}{2}|3\nabla u|^2 - u\,dx\, :\, u \geq \varphi \text{ in }\Omega  \right\},
\end{equation}
as $\eps\to 0^+$. Experiments with this example (without the obstacle) are also presented in \cite{armstrong2018iterative}.

We ran some experiments using the PDE acceleration method for solving this stochastic obstacle problem. Table \ref{tab:HG} shows the runtimes for different values of the damping parameter $a$. Figure \ref{fig:HG} shows a random checkerboard, the solution of the stochastic obstacle problem \eqref{eq:stob} and the solution of the homogenized problem \eqref{eq:stobhom}.  We mention that knowledge of the effective (homogenized) equation \eqref{eq:stobhom} can help with selecting the optimal damping parameter for the stochastic problem \eqref{eq:stob}. Indeed, by \eqref{eq:aopt} the optimal damping for the homogenized equation (without the obstacle) is $a=2\sqrt{3}\pi$, which is larger than the damping $a=2\pi$ we have been using in this paper so far. However, since the discussion in Section \ref{sec:optdamp} does not consider the obstacle, the true optimal damping parameter will depend on the smallest eigenvalue of the effective operator on the domain $\{u > \varphi\}$, which is initially unknown. Since this domain is strictly smaller than $\Omega$, monotonicity of eigenvalues implies that the optimal damping is larger than our computed $a=2\sqrt{3}\pi$. We find (see Table \ref{tab:HG}) that $a=6\pi$ is close to optimal for this problem. 

\begin{figure}
\centering
\subfloat[Checkerboard]{\includegraphics[trim = 65 30 70 20, clip=true,width=0.50\textwidth]{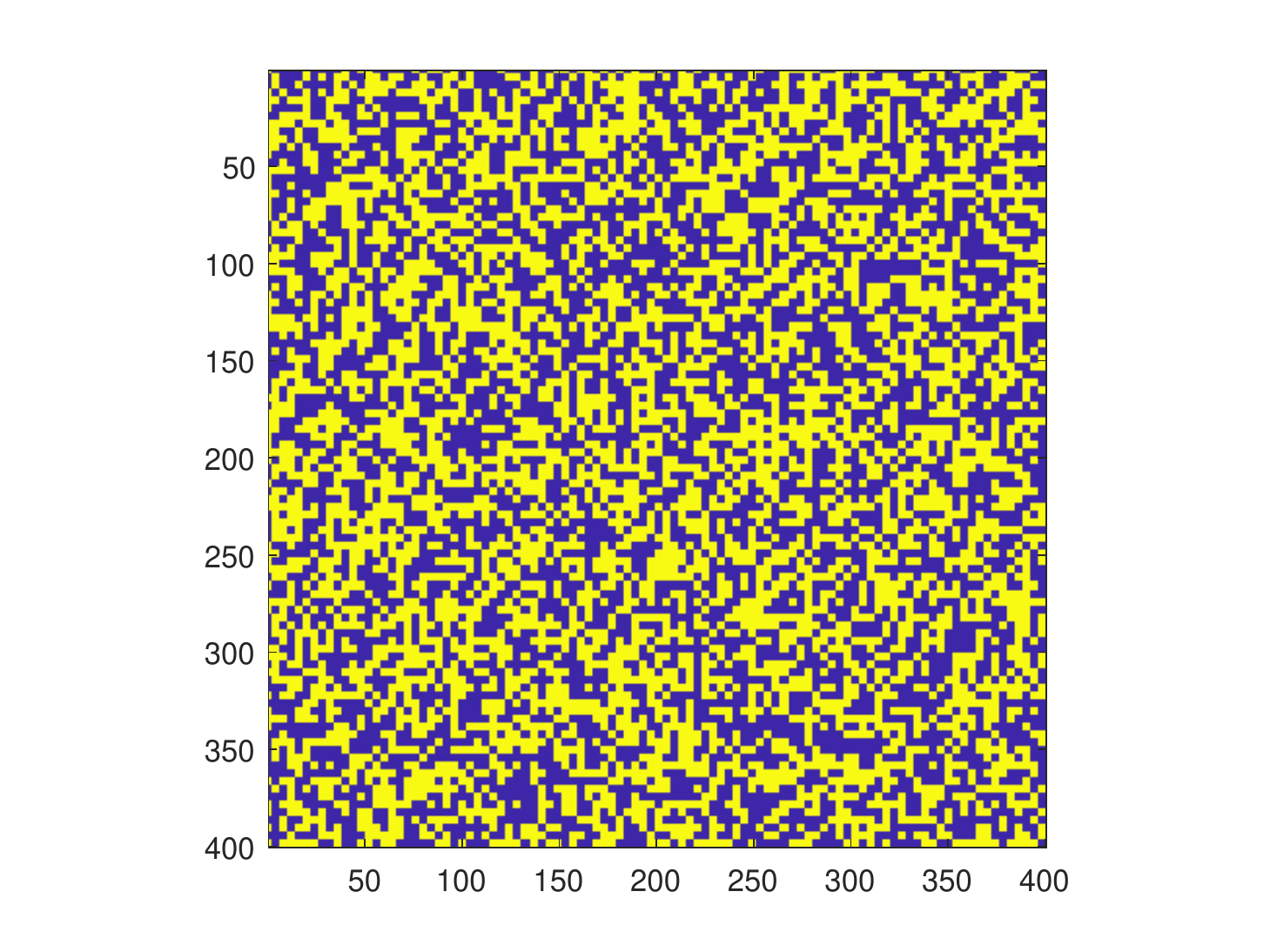}}
\subfloat[Solution of \eqref{eq:stob}]{\includegraphics[trim = 100 30 90 90, clip=true,width=0.50\textwidth]{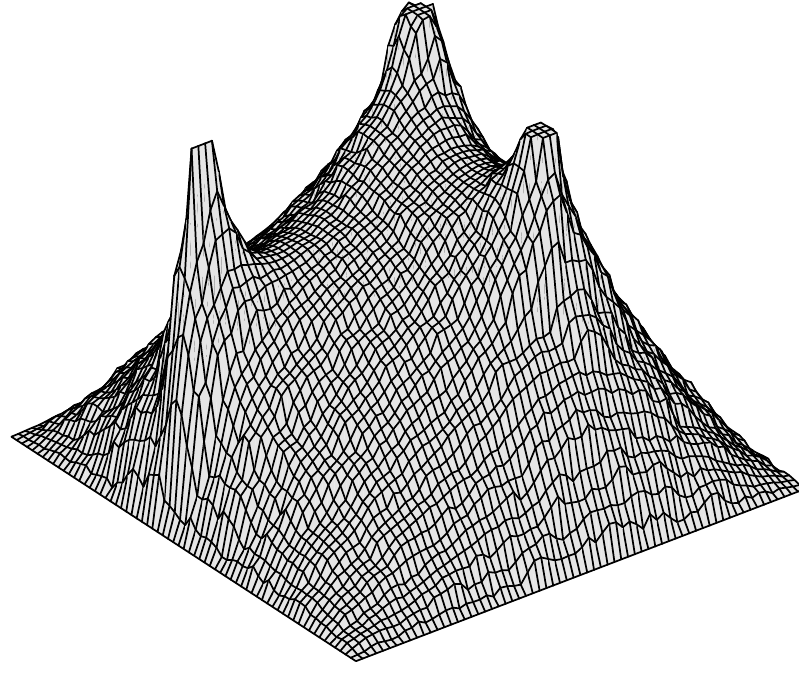}}

\subfloat[Contours of solution of \eqref{eq:stob}]{\includegraphics[trim = 65 30 70 20, clip=true,width=0.50\textwidth]{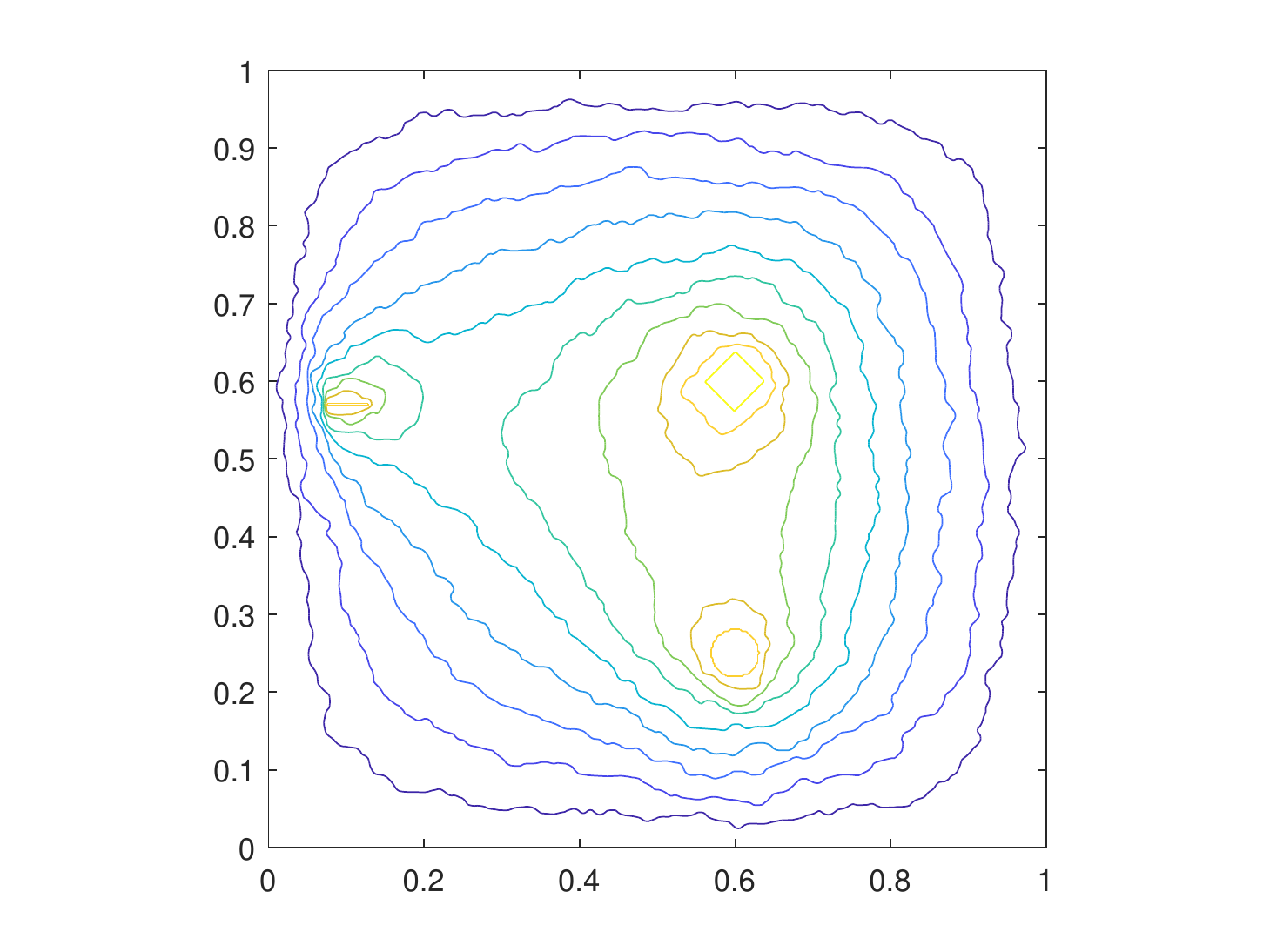}}
\subfloat[Contours of solution of \eqref{eq:stobhom}]{\includegraphics[trim = 65 30 70 20, clip=true,width=0.50\textwidth]{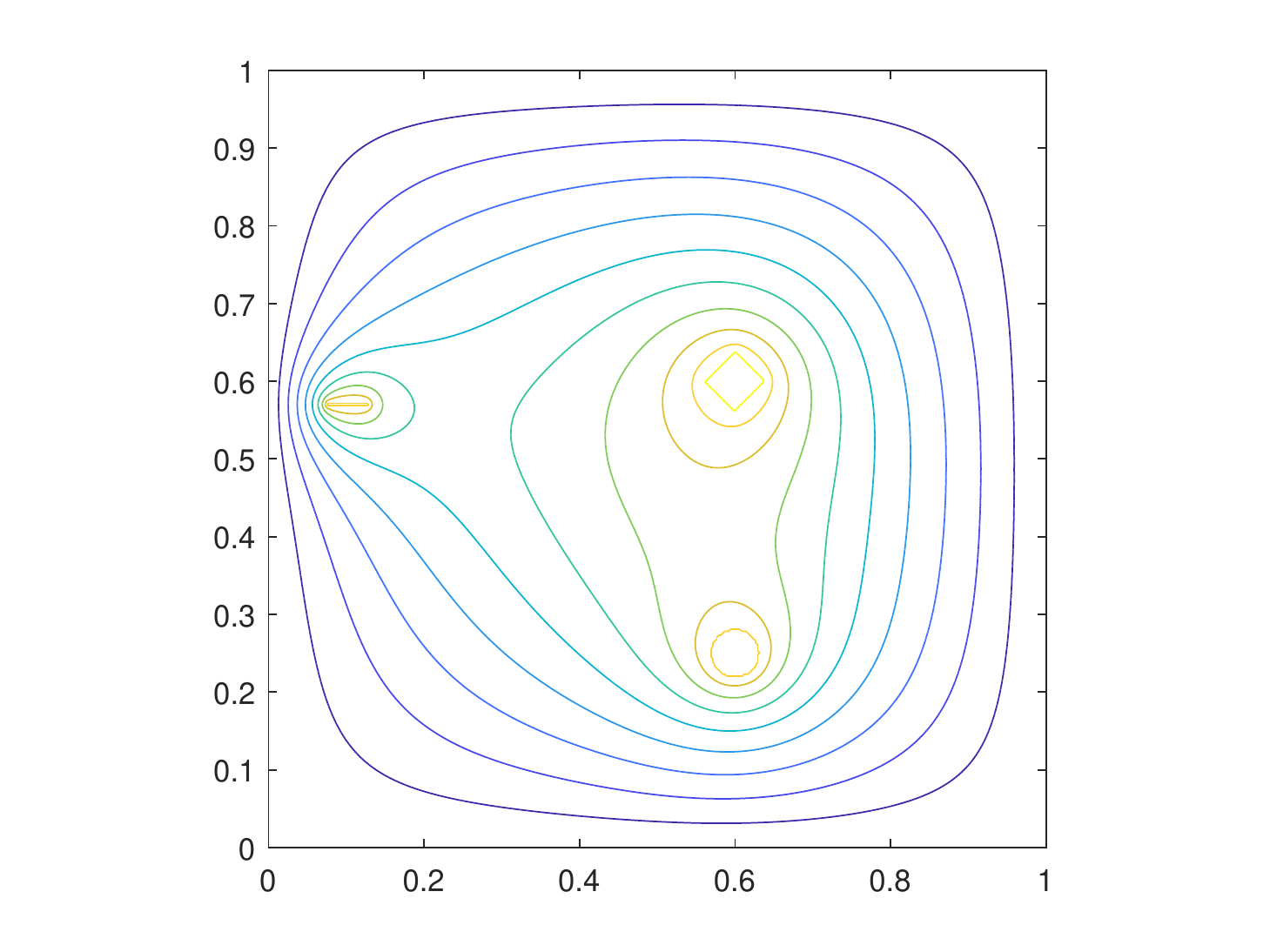}}
\caption{Solution and contours of the stochastic obstacle problem \eqref{eq:stob} and the effective equation \eqref{eq:stobhom}.}
\label{fig:HG}
\end{figure}

\begin{table}[!t]
\centering
\begin{tabular}{|c|c|c|c|c|c|c|c|}
 \hline \multicolumn{2}{|r|}{\bf Damping:} &\multicolumn{2}{c|}{$a=2\pi$}&\multicolumn{2}{c|}{$a=6\pi$}&\multicolumn{2}{c|}{$a=9\pi$}\\
\hline
Checkerboard & Mesh & Time (s) & Iter. & Time (s) & Iter. &Time (s) & Iter. \\
\hline
$16\times 16$ & $64^2$     & 0.037 	& 1665  & 0.012 	& 572   & 0.012 	& 569   \\ 
$32\times 32$ & $128^2$    & 0.315 	& 3924  & 0.115 	& 1340  & 0.12 	& 1469 \\ 
$64\times 64$ & $256^2$    & 3.35 	& 8919  & 1.16 	& 3087  & 1.34 	& 3588     \\ 
$128\times 128$& $512^2$   & 31.2 	& 20224 & 9.91 	& 6908  & 10.7 	& 7482     \\ 
$256\times 256$& $1024^2$  & 339 	& 45003 & 109 	   & 15197 & 118 	   & 16425  \\ 
\hline
\multicolumn{2}{|c|}{Complexity}  & 1.65 & 0.59 & 1.64 & 0.59 & 1.65 & 0.6 \\
\hline
\end{tabular}
\caption{Run times for the PDE accelerated solver on the stochastic homogenization obstacle problem with damping parameters $a=2\pi,6\pi,9\pi$.}
\label{tab:HG}
\end{table}

\section{Conclusion}

 We studied the recently introduced variational framework, called PDE acceleration, for applying accelerated gradient descent (or momentum descent) to problems in the calculus of variations. For a large class of convex optimization problems, the descent equations for PDE acceleration correspond to a nonlinear damped wave equation, which can be solved by a simple explicit forward Euler scheme. The acceleration is realized as a relaxation of the CFL condition for a wave equation ($dt\sim dx$) compared to a diffusion equation ($dt\sim dx^2$). We proved convergence with a linear rate for this class of accelerated PDE's, and applied the method to minimal surface obstacle problems, including a double obstacle problem with forcing, and a stochastic homogenization problem with obstacle constraint. In every case, PDE acceleration is faster than existing state of the art methods.

We mention briefly some ideas for future work. First, we use the damping parameter $a=2\pi$ throughout the whole paper, which is surely not optimal for every problem. We can achieve faster convergence for many experiments in the paper by hand tuning the damping. The difficulty with selecting the optimal damping is that it depends on the first Dirichlet eigenvalue (in the linear case) on the free boundary domain $\{u>\varphi\}$, which is \emph{a priori} unknown. A way to improve performance could be to compute the solution first on a coarse grid, and then estimate the optimal damping from the computed free boundary and use the optimal damping parameter when solving the equation on a finer mesh. 

There are other natural ways to speed up convergence, such as considering a multi-grid approach, or varying the damping parameter over time. The damping parameter controls the damping profile in the frequency domain; larger choices of the damping parameter give preference to damping higher frequencies at the expense of leaving lower frequencies underdamped. This is reminiscent of how the choice of grid resolution affects the damping in multi-grid methods, and a smart choice of a schedule for varying the damping parameter may result in a significant speed up.

Finally, the methods here are not restricted to second order equations, and can be applied almost directly to higher order equations, such as the fourth order PDEs that have proven popular in image processing \cite{tai2011fast,zhu2012image,you2000fourth}. In this case, PDE acceleration will relax the very stiff CFL condition ($dt\sim dx^4$) for fourth order equations to $dt\sim dx^2$. It is also possible to make other choices for the kinetic energy, which would lead to other flows that may be of interest (however, due to Ostrogradsky instability \cite{woodard2015theorem}, the kinetic energy should only contain first derivatives of $u$ in time). Problems in the calculus of variations arise in virtually all fields of science and engineering, include problems like image segmentation and noise removal \cite{rudin1992nonlinear,chan2001active}, minimal surfaces \cite{caffarelli1998obstacle}, and materials science \cite{ball1998calculus}, among many others. The results of this paper suggest that PDE acceleration can be a useful tool for solving many of these other problems, and we intend to pursue such applications, and others, in future work.

\end{document}